\documentclass[11pt,reqno]{amsart}
\usepackage {amssymb}
\usepackage {amsmath}
\usepackage {bbm}
\usepackage{amsthm}
\usepackage{mathtools}
\usepackage{graphicx}
\usepackage {amscd}
\usepackage {epic}
\usepackage{array}

\usepackage{mathrsfs}

\usepackage{etaremune}

\DeclareFontFamily{U}{dutchcal}{\hyphenchar\font=-1}
\DeclareFontShape{U}{dutchcal}{m}{n}{ <-> dutchcal-r}{}
\DeclareSymbolFont{dutchletters}{U}{dutchcal}{m}{n}

\DeclareMathSymbol{\mathdutchcal}{0}{dutchletters}{"41} 
\newcommand{\dcal}[1]{\text{\usefont{U}{dutchcal}{m}{n}#1}}

\usepackage[dvipsnames]{xcolor}
\usepackage[colorlinks,citecolor=OliveGreen,linkcolor=Mahogany,urlcolor=Plum,pagebackref]{hyperref}
\usepackage[alphabetic]{amsrefs}
\usepackage{cleveref}

\usepackage{enumerate}
\usepackage{tikz}
\usepackage{tikz-cd}
\usetikzlibrary{arrows.meta}
\usepackage{verbatim,color,geometry}
\usepackage[all]{xy}
\usepackage{enumitem}
\usepackage{bm}
\usepackage{longtable}
\usepackage{aliascnt}
\usetikzlibrary{calc}
\usepackage{textcomp}
\usepackage{tablefootnote}
\usepackage{float}

\newcommand{\mtf}[1]{\mathfrak{#1}}
\newcommand{\mts}[1]{\mathscr{#1}}
\newcommand{\wh}[1]{\widehat{#1}}
\newcommand{\ove}[1]{\overline{#1}}
\newcommand{\wt}[1]{\widetilde{#1}}

\newcommand{\cD}{\dcal{D}}
\newcommand{\cE}{\dcal{E}}
\newcommand{\cF}{\dcal{F}}
\newcommand{\cG}{\dcal{G}}
\newcommand{\cH}{\dcal{H}}
\newcommand{\cI}{\dcal{I}}

\newcommand{\cL}{\dcal{L}}

\newcommand{\cO}{\dcal{O}}
\newcommand{\cP}{\dcal{P}}

\newcommand{\cZ}{\dcal{Z}}

\newcommand\MM{{\mathcal{M}}}

\newcommand{\bP}{\mathbb{P}}
\newcommand{\bC}{\mathbb{C}}

\newcommand{\bQ}{\mathbb{Q}}
\newcommand{\bZ}{\mathbb{Z}}

\newcommand{\bV}{\mathbb{V}}
\newcommand{\bG}{\mathbb{G}}

\newcommand{\bN}{\mathbb{N}}

\newcommand{\Bl}{\mathrm{Bl}}
\newcommand{\NS}{\mathrm{NS}}
\newcommand{\Id}{\mathrm{Id}}

\newcommand{\Gr}{\mathrm{Gr}}

\newcommand{\Sym}{\mathrm{Sym}}

\newcommand{\sE}{\mathscr{E}}

\newcommand{\sP}{\mathscr{P}}

\newcommand{\sR}{\mathscr{R}}
\newcommand{\sS}{\mathscr{S}}
\newcommand{\sU}{\mathscr{U}}
\newcommand{\sV}{\mathscr{V}}

\newcommand{\sX}{\mathscr{X}}
\newcommand{\sY}{\mathscr{Y}}
\newcommand{\sZ}{\mathscr{Z}}

\DeclareMathOperator{\Aut}{Aut}
\DeclareMathOperator{\Stab}{Stab}

\DeclareMathOperator{\vol}{vol}

\DeclareMathOperator{\Bs}{Bs}

\DeclareMathOperator{\Hilb}{Hilb}

\DeclareMathOperator{\Spec}{Spec}

\DeclareMathOperator{\mult}{mult}

\DeclareMathOperator{\Pic}{Pic}

\DeclareMathOperator{\coeff}{coeff}
\DeclareMathOperator{\id}{id}

\DeclareMathOperator{\ord}{ord}

\DeclareMathOperator{\Proj}{Proj}

\DeclareMathOperator{\Supp}{Supp}
\DeclareMathOperator{\diag}{diag}

\DeclareMathAlphabet{\mathbbb}{U}{bbold}{m}{n}

\newcommand{\mb}[1]{\mathbb{#1}}

\newcommand{\SL}{\mathrm{SL}}

\newcommand{\PGL}{\mathrm{PGL}}
\newcommand{\CM}{\mathrm{CM}}

\newcommand{\RHom}{\mathrm{RHom}}

\newcommand{\GIT}{\mathrm{GIT}}

\newcommand{\sF}{\mathscr{F}}

\newcommand{\fM}{\mathfrak{M}}

\newcommand{\sC}{\mathscr{C}}

\newcommand{\K}{\mathrm{K}}

\newcommand{\Def}{\mathrm{Def}}
\newcommand{\sst}{\mathrm{ss}}

\newcommand{\Ext}{\mathrm{Ext}}
\newcommand{\plt}{\mathrm{plt}}

\newcommand{\sslash}{\mathbin{/\mkern-6mu/}}

\numberwithin{equation}{section}

\newtheorem{prop}{Proposition}[section]

\newcommand{\newaliastheorem}[3]{%
  \newaliascnt{#1}{prop}%
  \newtheorem{#1}[#1]{#2}%
  \aliascntresetthe{#1}%
  \crefname{#1}{#2}{#3}%
  \Crefname{#1}{#2}{#3}%
}

\newaliastheorem{thm}{Theorem}{Theorems}
\newaliastheorem{theorem}{Theorem}{Theorems}
\newaliastheorem{lem}{Lemma}{Lemmas}
\newaliastheorem{lemma}{Lemma}{Lemmas}
\newaliastheorem{cor}{Corollary}{Corollaries}
\newaliastheorem{corollary}{Corollary}{Corollaries}
\newaliastheorem{theodef}{Theorem-Definition}{Theorem-Definitions}
\newaliastheorem{prop-def}{Proposition-Definition}{Proposition-Definitions}
\newaliastheorem{convention}{Convention}{Conventions}
\newaliastheorem{conj}{Conjecture}{Conjectures}
\newaliastheorem{conjecture}{Conjecture}{Conjectures}

\theoremstyle{definition}
\newaliastheorem{defi}{Definition}{Definitions}
\newaliastheorem{defn}{Definition}{Definitions}
\newaliastheorem{exa}{Example}{Examples}
\newaliastheorem{exam}{Example}{Examples}
\newaliastheorem{expl}{Example}{Examples}
\newaliastheorem{rem}{Remark}{Remarks}
\newaliastheorem{rmk}{Remark}{Remarks}
\newaliastheorem{remark}{Remark}{Remarks}
\newaliastheorem{que}{Question}{Questions}

\crefname{prop}{Proposition}{Propositions}
\Crefname{prop}{Proposition}{Propositions}

\title{Degenerations of elliptic quartics and K-moduli of Fano threefolds}

\author{Ivan Cheltsov}
\address{School of Mathematics, University of Edinburgh, James Clerk Maxwell Building, Peter Guthrie Tait Road, Edinburgh, EH9 3FD, Scotland, UK}
\email{i.cheltsov@ed.ac.uk}

\author{Anne-Sophie Kaloghiros}
\address{Department of Mathematics, Brunel University London, Uxbridge UB8 3PH, United Kingdom}
\email{anne-sophie.kaloghiros@brunel.ac.uk}

\author{Robert \'{S}miech}
\address{School of Mathematics, University of Edinburgh, James Clerk Maxwell Building, Peter Guthrie Tait Road, Edinburgh, EH9 3FD, Scotland, UK}
\email{robert.smiech@ed.ac.uk}

\author{Junyan Zhao}
\address{Department of Mathematics, University of Maryland, William E. Kirwan Hall, 4176 Campus Dr, College Park, MD 20742, USA}
\email{jzhao81@umd.edu}

\date{}

\begin{document}

\begin{abstract}
We study the K-moduli space of Fano threefolds obtained by first blowing up $\bP^3$ along an elliptic quartic curve $C$ and then blowing up a fiber $\ell_p$ of the exceptional divisor $E \to C$. We prove that this K-moduli space is isomorphic to a VGIT quotient parametrizing pairs $(C,p)$, with linearization induced by the CM line bundle. In particular, we classify all K-(semi/poly)stable members of this deformation family. The main new ingredients in the proof include the geometry of the Hilbert scheme of elliptic quartic curves, deformation theory of Fano--K3 pairs, and optimal bounds on the local volumes of threefold singularities.
\end{abstract}
\maketitle

\tableofcontents

\section{Introduction}

K-stability originated in differential geometry (cf.\ \cites{Tia97,Don02}) as a criterion
for the existence of K\"ahler--Einstein metrics on Fano varieties (cf.\ \cites{CDS15,LXZ22}).
Its algebraic reformulation has since developed into an independent theory, whose central
achievement is the construction of K-moduli spaces parametrizing K-polystable Fano
varieties; see e.g.\ \cite{Xu25}.

Fano threefolds provide a host of concrete examples against which the general theory can be
tested. Studying the K-moduli space is not only interesting on its own; it also gives an
effective, geometric way to understand K-stability, rather than checking it case by case.
The main approach to the K-moduli of Fano threefolds is the moduli continuity method
(cf.\ \cites{LX19,LZ25,Zha26}), which so far has been applied mainly to del Pezzo threefolds
and to Fano threefolds of Picard rank two, typically blowups of a Fano along a curve in
which either the Fano or the curve is rigid. Neither restriction is incidental: high Picard
rank makes it hard to track several divisors at once, while a complicated blowup center
produces degenerations that may change the type of the limit. The moduli continuity method
rests on two inputs: the moduli of anticanonical K3 surfaces and local volume bounds on
singularities. Recent work \cites{Liu25,KLPZ26} advances both: it sharpens the local volume
bounds for singularities of K-semistable Fano threefolds and studies Fano--K3 pairs from a
new deformation-theoretic perspective, providing a general framework for Fano threefolds of
large volume.

\smallskip
In this paper, we put these tools to work, extending the moduli continuity method to a
family of Picard rank three: family \textnumero3.11 in the Iskovskikh--Mori--Mukai
classification, obtained by blowing up $\bP^3$ along a quartic elliptic curve, i.e.\ a
$(2,2)$-complete intersection, and then blowing up one fiber of the exceptional divisor. In
doing so, we go beyond a routine application of the existing framework: we generalize the
original approach to the moduli of K3 surfaces so as to handle higher Picard ranks, turn the
general structural results of \cites{Liu25,KLPZ26} into concrete computations, and introduce
the method of Hilbert schemes. This family is of independent interest: Fujita \cite{Fuj23}
showed that a general member of \textnumero3.11 is K-stable by verifying the K-stability of
an explicit member via admissible flags and equivariant K-stability (cf.\ \cites{Zhu21,AZ22}); this is a long and carefully executed computation, reflecting the subtlety of verifying
K-stability directly.

\smallskip

In view of how the threefolds in this family are constructed, it is natural to compare the
K-moduli with a GIT quotient of the space of pairs $(C,p)$, where $C$ is a
$(2,2)$-complete intersection curve and $p\in C$. For the purpose of GIT, one may
equivalently work with the pairs $(Y,\ell_p)$, where $Y\coloneqq\Bl_C\bP^3$ is a
hypersurface of bidegree $(1,2)$ in $\bP^1\times\bP^3$, and $\ell_p$ is the fiber of the
exceptional divisor of $Y\to\bP^3$ over $p\in C$, i.e.\ the fiber of the projection
$\pi_2\colon\bP^1\times\bP^3\to\bP^3$ over $p$ that is contained in $Y$. These pairs are
parametrized by a projective bundle $\bP\cE$ over $\bP^3$; since $\bP\cE$ has Picard rank
two, the polarization $\cO_{\bP\cE}(1)+t\pi_2^*\cO_{\bP^3}(1)$ defining the GIT quotient
depends on a single parameter $t$. Denote by
$\ove{\fM}^{\GIT}(t)=\bP\cE\sslash_t\PGL_2\times\PGL_4$ the corresponding quotient (cf. \Cref{sec:GIT}).

\smallskip

Let $\MM^{\K}_{\textup{\textnumero3.11}}$ be the K-moduli stack of family \textnumero3.11,
an irreducible component of the K-moduli stack $\MM^\K_{3,28}$ of Fano threefolds of
anticanonical volume $28$. Our main result shows that every member of
$\MM^{\K}_{\textup{\textnumero3.11}}$ is isomorphic to the blowup of $\bP^3$ along a
$(2,2)$-complete intersection curve, followed by the blowup of a fiber of the exceptional
divisor. Moreover, we identify this K-moduli space with one of the quotients
$\ove{\fM}^{\GIT}(t)$.

 \begin{thm}\label{thm:main-thm}   
 Let $\MM^{\K}_{\textup{\textnumero3.11}}$ be the K-moduli stack of family \textup{\textnumero3.11}, and let $\ove{\fM}^{\K}_{\textup{\textnumero3.11}}$ be its good moduli space. Then the following hold.    \begin{enumerate}        
 \item[\textup{(1)}] Every K-semistable Fano variety $X$ parametrized by        $\MM^{\K}_{\textup{\textnumero3.11}}$ is K-stable, and is isomorphic to the blowup of $Y\coloneqq\Bl_C\bP^3$, for some $(2,2)$-complete intersection curve $C\subseteq\bP^3$, along a fiber $\ell$ of the exceptional divisor over $C$; here $Y$ is generically smooth along $\ell$.        \item[\textup{(2)}] The stack $\MM^{\K}_{\textup{\textnumero3.11}}$ is a smooth connected Deligne--Mumford component of $\MM^{\K}_{3,28}$,  and $\ove{\fM}^{\K}_{\textup{\textnumero3.11}}$ is normal with quotient singularities.        \item[\textup{(3)}] There is a natural isomorphism      \[\ove{\fM}^{\K}_{\textup{\textnumero3.11}}\ \simeq \ \ove{\fM}^{\GIT}(\tfrac{5}{21}),\] which fits into the wall-crossing diagram    \begin{equation}\nonumber    \xymatrix @R=.07in @C=.07in{      &   & &      \ove{\fM}^{\GIT}(0,\tfrac{2}{7})      \ar[rdd]_{\iota}\ar[ldd]_{\phi}\ar@{-->}[rr]      & &      \ove{\fM}^{\GIT}(\tfrac{2}{7},\tfrac{1}{3})      \ar[ldd]^{\psi}\ar[rdd]^{\rho}      & \\      &&&&&& \\      & &      \bP^1\simeq\ove{\fM}^{\GIT}(0)      & &      \ove{\fM}^{\GIT}(\tfrac{2}{7})      & &      \ove{\fM}^{\GIT}(\tfrac{1}{3})=\Spec\bC,    }    \end{equation}    
 in which \begin{itemize}        
 \item $\phi$ is generically a $\bP^1$-fibration;   \item $\iota$ contracts the curve in $\ove{\fM}^{\GIT}(0,\tfrac27)$ parametrizing        those pairs $(Y,\ell_p)$ for which $p$ is an inflection point of $C_Y$; and        
 \item $\psi$ is an isomorphism.    
 \end{itemize}    
 \end{enumerate}
 \end{thm} 

As a consequence, every smooth member of the deformation family \textnumero3.11 is
K-stable, and we classify all K-semistable limits of the family by studying the variation
of GIT quotients $\ove{\fM}^{\GIT}(t)$.

\begin{thm}\label{thm:equivalence}    Let $C\subseteq\bP^3$ be a $(2,2)$-complete intersection curve, let $p\in C$ be a    point, and set $X\coloneqq\Bl_{\ell_p}\big(\Bl_C\bP^3\big)$, where $\ell_p$ is the fiber of the exceptional divisor over $p$. Then the following are equivalent:    \begin{enumerate}        \item[\textup{(1)}] $X$ is K-semistable;        \item[\textup{(2)}] $X$ is K-stable;        \item[\textup{(3)}] $C$ is either integral with at worst nodal singularities, or        the union of a twisted cubic and a line meeting nodally at two distinct points; moreover, in either case $p$ is a smooth point of $C$ lying on a component of degree at least $3$.    \end{enumerate}    In particular, every K-semistable Fano threefold in the deformation family    \textup{\textnumero3.11} has at worst ordinary double points.\end{thm}

Since the threefolds in family \textnumero3.11 are themselves blowups of those in family
\textnumero2.25, the blowup of $\bP^3$ along a $(2,2)$-complete intersection curve, the arguments used to prove \Cref{thm:main-thm} and \Cref{thm:equivalence} apply verbatim to
family \textnumero2.25 and give a complete description of its K-moduli space. This moduli
space is one-dimensional and was previously determined by direct computation in
\cites{Pap22,ACD+23}.

\subsection*{Sketch of proofs}

We now sketch the proof of \Cref{thm:main-thm}, especially part~(1), which proceeds in two
steps.

\smallskip
For a K-semistable Fano $X\in\MM^\K_{\textup{\textnumero3.11}}$, the first step toward
\Cref{thm:main-thm}(1) is to construct a birational contraction $X\dashrightarrow\bP^3$.
Following the idea of \cite[Section~4]{LZ25}, we take a one-parameter family
$\sX\to(0\in T)$ degenerating a smooth member of family \textnumero3.11 to $X=\sX_0$, and
consider the degeneration of the divisors $\cH_t$ on $\sX_t$ that induce the blowups
$\sX_t\to\bP^3$. In contrast to \cite{LZ25}, which applies the theory of moduli of
lattice-polarized K3 surfaces directly, we use this theory more implicitly: we translate
the nefness of $\cH_t$ into the non-existence of integral solutions to certain inequalities
imposed by the Hodge index theorem. This is needed to overcome a technical difficulty which
arises once the Picard rank is at least three: the $\bQ$-Cartierness of $-K_{\sX}$ and
$\cH$ does not imply that of the other divisors on $\sX$. Notably, applying this idea in the
setting of \cite{LZ25} also simplifies that part of the argument there considerably. The
outcome, after shrinking the base $T$, is a contraction from a small crepant model
$\wt{\sX}\to\sX$ to $\bP^3_T$ whose generic fiber $\wt{\sX}_t=\sX_t\to\bP^3_t$ is the blowup morphism;
see \Cref{prop:contraction of E}.

\smallskip
The genuinely new difficulty is the second step: showing that
$\wt{X}\coloneqq\wt{\sX}_0$ is isomorphic to $X$ and that $\wt{\sX}_0\to\bP^3$ is indeed the two-step
blowup. To this end, we take the closure in $\wt{\sX}$ of the divisor $\cE_t$, the
exceptional divisor of $\sX_t\to\bP^3$ dominating the $(2,2)$-complete intersection
curve $C_t$, and analyze the central fiber $C\coloneqq\sC_0$ of its image $\sC$ under
$\wt{\sX}\to\bP^3_T$. In \cite{LZ25,Zha26} the center of the relevant blowup is a point or a
line, both of which are rigid, so the limiting center cannot change type. For family
\textnumero3.11, by contrast, $C$ is only known to be a flat degeneration of
$(2,2)$-complete intersection curves in $\bP^3$, and this is where the Hilbert scheme comes
in. We classify all possible limits (cf.\ \Cref{lem:Hilbert-scheme}) and then rule out every
one that is not a complete intersection, each by a different mechanism. Concretely, such a
$C$ is one of the following:
\begin{itemize}
    \item a \emph{singular plane quartic with embedded points}. The pathology here is that
    $\wt{X}$ contains a surface contracted to a point by the crepant birational morphism
    $\wt{X}\to X$; this is a statement about singularities, excluded using \cite{Liu25}
    (cf.\ \Cref{lem:not-plane-curve}).
 \item a \emph{plane cubic together with a line meeting it nodally}. Here the obstruction is
global: this limit exists but the resulting Fano threefolds are K-unstable. We show that in this case the anticanonical K3 surfaces of $X$ lie on a certain Noether--Lefschetz divisor in the moduli of lattice-polarized K3 surfaces, and that every
Gorenstein canonical Fano threefold whose anticanonical K3 surface lies on this divisor is
K-unstable (cf.\ \Cref{prop:K-unstable-I}). This is established by studying the forgetful
map from the moduli of Fano--K3 pairs $(X,S)$ to the moduli of K3 surfaces $(S,-K_X|_S)$, and it is the deformation theory in \cite{KLPZ26} that makes this argument
work.
    \item an \emph{ACM non-reduced curve contained in a double plane}. This is in fact a
    degeneration of the previous case, and is ruled out by the same argument.
    \item \emph{four lines through a common point}. Although this is a complete
    intersection, it becomes a pathology when the fiber to be blown up lies over the common
    point. The obstruction here is birational-geometric: the center has a non-planar
    singularity, and blowing it up produces non-isolated singularities along a curve. This
    is excluded by an explicit local computation together with a birational argument.
\end{itemize}
Once this is done, \Cref{thm:main-thm}(1) follows quickly. The pathologies above
require three distinct tools; in particular, no single computation handles all of them,
which is why the direct approach of the earlier papers does not extend.

\subsubsection*{Outline of the paper}

\Cref{prelim} collects the basic properties of smooth
members of family \textnumero3.11, together with the background on K-stability and K-moduli
that we need. \Cref{sec:3} contains the heart of the paper, the proof of
\Cref{thm:main-thm}(1). \Cref{sec:GIT} is independent of K-stability: there we work out the
variation of GIT quotients $\ove{\fM}^{\GIT}(t)$ and its wall-crossing structure. Finally,
\Cref{sec:Kmod} combines these two threads to prove the remaining parts of
\Cref{thm:main-thm} and \Cref{thm:equivalence}.

\subsection*{AI usage disclosure}
OpenAI Codex was only used for editorial assistance in preparing this manuscript and proofread. All mathematical content was developed and verified by the authors, who take full responsibility for the content and correctness of the paper.

\subsection*{Acknowledgments}
We are grateful to Yuchen Liu, Kento Fujita and Alan Thompson for fruitful discussions. This collaboration began with the authors' meeting in De Morgan House in February 2025, which was funded by a Focused Research Grant of the Heilbronn Institute for Mathematical Research. IC and R\'S are supported by the Simons Collaboration Grant \emph{Moduli of varieties}. JZ is supported by Maryland postdoc travel grant. ASK was supported by EPSRC grant EP/V056689.

\section{Preliminaries}\label{prelim}

\subsection*{Conventions}

\begin{itemize}
    \item We work over the field $\bC$ of complex numbers.
    \item For any vector space $V$, the projectivization $\bP V$ is $\Proj \big(\bigoplus_{d\geq 0}\Sym^d V\big)$, which parametrizes all $1$-dimensional \emph{quotients} of $V$. 
    \item For any scheme $X$ and any locally free sheaf $\cE$ on $X$, the projective bundle $\bP(\cE)$ is $\Proj_X \big(\bigoplus_{d\geq 0}\Sym^d \cE\big)$. We do not distinguish a locally free sheaf with the corresponding vector bundle.
\end{itemize}

\subsection{Smooth Fano threefolds \textnumero 3.11}\label{subsection:3.11}

Let $C=V(q_0,q_1)\subset\bP^3$ be a smooth complete intersection of two quadrics, and let $p\in C$.
The blowup of $\bP^3$ along $C$ admits the standard realization
\[
    Y\ \coloneqq\ \Bl_C\bP^3
   \  \simeq\ 
    \bigl\{
        ([s:t],x)\in\bP^1\times\bP^3
        \mid
        sq_1(x)-tq_0(x)=0
    \bigr\}.
\]
Thus, $Y$ is a hypersurface of bidegree $(1,2)$ in
$\bP^1\times\bP^3$. Since $p\in C$, its fiber
\[
    \ell_p\ \coloneqq\ 
    \bigl(\operatorname{pr}_2|_Y\bigr)^{-1}(p)
    \ \simeq\ \bP^1
\]
is a fiber of the exceptional divisor $E_Y\to C$. The corresponding
Fano threefold in Mori--Mukai family~\textnumero 3.11 is
\[
    X\simeq\Bl_{\ell_p}Y.
\]
Equivalently, if $\widetilde{C}\subset\Bl_p\bP^3$ denotes the strict
transform of $C$, then
\[
    X\simeq\Bl_{\widetilde{C}}\bigl(\Bl_p\bP^3\bigr).
\] As described in~\cite[\S III.3]{Mat95}, these two constructions, together
with the remaining elementary Mori contractions of $X$, fit into the
following commutative diagram:
\begin{center}
\begin{tikzpicture}[
    vertex/.style={minimum size=9mm, inner sep=1pt},
    centre/.style={minimum size=9mm, inner sep=1pt},
    arrow/.style={-{Stealth[length=3mm]}}
]
\def\r{2.5}

\foreach \i/\ang/\label in {
    1/90/$\bP^1\times\bP^2$,
    2/30/$\bP^2$,
    3/-30/$\Bl_p\bP^3$,
    4/-90/$\bP^3$,
    5/-150/$\Bl_C\bP^3$,
    6/150/$\bP^1$
} {
    \node[vertex] (v\i) at (\ang:\r) {\label};
}

\node[centre] (c) at (0,0) {$X$};

\draw[arrow] (c) -- node[midway, right] {$\sigma_3$} (v1);
\draw[arrow] (c) -- node[midway, above] {$\sigma_1$} (v3);
\draw[arrow] (c) -- node[midway, above] {$\sigma_2$} (v5);

\draw[arrow] (v1) -- node[midway, above] {$\operatorname{pr}_2$} (v2);
\draw[arrow] (v1) -- node[midway, above] {$\operatorname{pr}_1$} (v6);

\draw[arrow] (v3) -- (v2);
\draw[arrow] (v3) -- (v4);
\draw[arrow] (v5) -- (v4);
\draw[arrow] (v5) -- (v6);
\end{tikzpicture}
\end{center}
More precisely, $\sigma_1$ is the blowup of $\Bl_p\bP^3$ along
$\widetilde{C}$, while $\sigma_2$ is the blowup along the fiber $\ell_p$ of the second projection
$\Bl_C\bP^3\to\bP^3$. The third contraction $\sigma_3$ is the blowup of a smooth elliptic curve of bidegree $(2,3)$. Under
the morphism
\[
    X\ \stackrel{\sigma_1}{\longrightarrow}\ \Bl_p\bP^3\ \longrightarrow\ \bP^3,
\]
the exceptional divisor of $\sigma_3$ is the strict transform of the
cubic cone over $C$ with vertex $p$.

\smallskip

We use the presentation
\[
    \Pic(X)=\bZ H\oplus\bZ E\oplus\bZ F,
\]
where $H$ is the pullback of the hyperplane class on $\bP^3$, $E$ is
the strict transform on $X$ of the exceptional divisor of
$\Bl_C\bP^3\to\bP^3$, and $F$ is the exceptional divisor of
$\sigma_2$. The canonical divisor formula gives
\[
    -K_X=4H-E-2F,
    \qquad
    (-K_X)^3=28.
\]

\subsection{K-stability and K-moduli theory}

\begin{defn}
A \emph{$\bQ$-Fano variety} (resp.\ \emph{weak $\bQ$-Fano variety})
is a normal projective variety $X$ with klt singularities such that
$-K_X$ is an ample (resp.\ big and nef) $\bQ$-Cartier divisor. For a weak $\bQ$-Fano variety $X$, the anticanonical divisor $-K_X$
is semiample by~\cite{BCHM}. Its ample model
\[
    \overline{X}
    \coloneqq
    \Proj R(X,-K_X)
\]
is a $\bQ$-Fano variety, called the \emph{anticanonical model} of $X$.
\end{defn}

\begin{defn}
A $\bQ$-Fano variety $X$ (resp.\ weak $\bQ$-Fano variety) is called
\emph{$\bQ$-Gorenstein smoothable} if there exist a pointed smooth curve
$(0\in T)$ and a projective flat morphism
\[
    \pi\colon\sX\longrightarrow T
\]
satisfying the following conditions:
\begin{itemize}
    \item $K_{\sX/T}$ is $\bQ$-Cartier and
          $-K_{\sX/T}$ is $\pi$-ample
          (resp.\ $\pi$-big and $\pi$-nef);
    \item $\pi$ is smooth over $T^\circ\coloneqq T\setminus\{0\}$; and
    \item $\sX_0\simeq X$.
\end{itemize}
\end{defn}

\begin{defn}\label{defn:beta}
Let $X$ be an $n$-dimensional $\bQ$-Fano variety, and let $E$ be a
prime divisor on a normal variety $Y$ admitting a birational morphism
$\pi\colon Y\to X$. The \emph{log discrepancy} of $E$ with respect to
$X$ is
\[
    A_X(E)
    \ \coloneqq\ 
    1+\coeff_E\bigl(K_Y-\pi^*K_X\bigr).
\]
The \emph{$S$-invariant} and \emph{$\beta$-invariant} of $E$ are defined
respectively by
\[
    S_X(E)
    \ \coloneqq\ 
    \frac{1}{(-K_X)^n}
    \int_0^\infty
        \vol_Y\bigl(-\pi^*K_X-tE\bigr)\,dt
\]
and
\[
    \beta_X(E)
   \  \coloneqq\ 
    A_X(E)-S_X(E).
\]
\end{defn}

\begin{theodef}[ \textup{cf. \cite{Fuj19,Li17,BX19,LWX21}}]\label{theodef:K-stability}
A $\bQ$-Fano variety $X$ is
\begin{enumerate}
    \item[\textup{(1)}] \emph{K-semistable} if and only if $\beta_X(E)\geq 0$          for every prime divisor $E$ over $X$;

    \item[\textup{(2)}] \emph{K-stable} if and only if $\beta_X(E)>0$ for every prime divisor $E$ over $X$;

    \item[\textup{(3)}]  \emph{K-polystable} if it is K-semistable and every
          $\mb{G}_m$-equivariant K-semistable degeneration of $X$ is
          a product degeneration;
    \item[\textup{(4)}]  \emph{K-unstable} if it is not K-semistable.
\end{enumerate}
\end{theodef}

\begin{defn}
Let $x\in X$ be an $n$-dimensional klt singularity, and let
$E$ be a prime divisor over $X$ whose center on $X$ is $\{x\}$.
The \emph{volume} of $E$ at $x\in X$ is
\[
    \vol_{x,X}(E)
    \ \coloneqq\ 
    \lim_{m\to\infty}
    \frac{
        \dim_{\bC}
        \mathcal{O}_{X,x}/
        \{f\in\mathcal{O}_{X,x}\mid \ord_E(f)\geq m\}
    }{m^n/n!}.
\]
Its \emph{normalized volume} is
\[
    \widehat{\vol}_{x,X}(E)
   \  \coloneqq\ 
    A_X(E)^n\vol_{x,X}(E).
\]
The \emph{normalized volume} of the singularity $x\in X$ is
\[
    \widehat{\vol}(x,X)
    \ \coloneqq\ 
    \inf_E\widehat{\vol}_{x,X}(E),
\]
where the infimum is taken over all prime divisors $E$ over $X$ whose
center on $X$ is $\{x\}$.
\end{defn}

We next recall the CM line bundle associated with a flat family of
polarized varieties; see~\cite{PT06,PT09,Tia97}. Let $\pi\colon\sX\longrightarrow S$ be a projective flat morphism of connected schemes whose fibers are
$S_2$ and of pure dimension $n$, and let $\cL$ be a
$\pi$-ample line bundle on $\sX$. By~\cite{KM76}, there exist
line bundles
\[
    \lambda_i\ = \ \lambda_i(\sX,\cL)
\]
on $S$ such that, for every sufficiently large integer $k$,
\[
    \det\bigl(\pi_!\cL^{\otimes k}\bigr)
    \ = \
    \lambda_{n+1}^{\otimes\binom{k}{n+1}}
    \otimes
    \lambda_n^{\otimes\binom{k}{n}}
    \otimes\cdots\otimes
    \lambda_0^{\otimes\binom{k}{0}}.
\]
Write the Hilbert polynomial of the fibers as
\[
    \chi\bigl(\sX_s,\cL_s^{\otimes k}\bigr)
    \ =\ 
    b_0k^n+b_1k^{n-1}+O(k^{n-2}).
\]

\begin{defn}
The \emph{CM $\bQ$-line bundle} of the polarized family
$(\pi\colon\sX\to S,\cL)$ is
\[
    \lambda_{\CM,\pi,\cL}
    \coloneqq
    \lambda_{n+1}^{
        \otimes\left(n(n+1)+\frac{2b_1}{b_0}\right)
    }
    \otimes
    \lambda_n^{\otimes(-2(n+1))}.
\]
Suppose, in addition, that both $\sX$ and $S$ are normal and
that $-K_{\sX/S}$ is a $\pi$-ample $\bQ$-Cartier divisor.
Choose $l\in\bZ_{>0}$ such that
\[
    \cL=-lK_{\sX/S}
\]
is a line bundle. We then define
\[
    \lambda_{\CM,\pi}
    \coloneqq
    \lambda_{\CM,\pi,\cL}^{\otimes 1/l^n}.
\]
This definition is independent of the choice of $l$.
\end{defn}
\begin{theorem}[K-moduli theorem]
\label{thm:K-moduli theorem}
Fix $n\in\bN$ and $v\in\bQ_{>0}$. Let $\MM^\K_{n,v}$ be the moduli
functor assigning to each scheme $T$ the groupoid
\[
\left\{
\sX/T
\ \middle|\
\begin{array}{l}
\sX\to T \text{ is a proper flat family satisfying Koll\'ar's
condition, and}\\
\text{every geometric fiber }\sX_{\bar t}\text{ is a K-semistable
$\bQ$-Fano variety}\\
\text{of dimension $n$ and anticanonical volume $v$}
\end{array}
\right\}.
\]
Then $\MM^\K_{n,v}$ is represented by an Artin stack, still denoted by
$\MM^\K_{n,v}$, which is of finite type over $\bC$ and has affine
diagonal. Its $\bC$-points parametrize K-semistable $\bQ$-Fano
varieties of dimension $n$ and anticanonical volume $v$. Moreover, $\MM^\K_{n,v}$ admits a good moduli space $\overline{\fM}^\K_{n,v}$,
which is a projective scheme whose closed points parametrize
K-polystable $\bQ$-Fano varieties of dimension $n$ and anticanonical
volume $v$. Finally, the CM $\bQ$-line bundle $\lambda_{\CM}$ on
$\MM^\K_{n,v}$ descends to an ample $\bQ$-line bundle
$\Lambda_{\CM}$ on $\overline{\fM}^\K_{n,v}$.
\end{theorem}

\begin{defn}
Let $\MM^{\textup{K}}_{\textup{\textnumero 3.11}}$
(resp.\ $\overline{\fM}^{\textup{K}}_{\textup{\textnumero 3.11}}$)
be the irreducible component of
$\MM^{\textup{K}}_{3,28}$
(resp.\ $\overline{\fM}^{\textup{K}}_{3,28}$), endowed with its
reduced stack (resp.\ scheme) structure, whose general point
parametrizes a smooth K-stable Fano threefold in deformation
family~$\textup{\textnumero 3.11}$. We call
$\MM^{\textup{K}}_{\textup{\textnumero 3.11}}$ and
$\overline{\fM}^{\textup{K}}_{\textup{\textnumero 3.11}}$
the \emph{K-moduli stack} and the \emph{K-moduli space},
respectively, of Fano threefolds in family~$\textup{\textnumero 3.11}$.
\end{defn}

\begin{remark}
\begin{enumerate}
    \item By~\cite[Theorem~1.1]{Fuj23}, deformation
          family~$\textup{\textnumero 3.11}$ contains a smooth
          K-stable Fano threefold. Consequently, its K-moduli stack
          and K-moduli space are nonempty and contain a dense open
          locus parametrizing K-stable members.

    \item The ambient K-moduli stack
          $\MM^{\textup{K}}_{3,28}$ and its good moduli space
          $\overline{\fM}^{\textup{K}}_{3,28}$ need not, a priori,
          be reduced, irreducible, or connected. Thus, the preceding
          definition alone does not ensure that
          $\MM^{\textup{K}}_{\textup{\textnumero 3.11}}$ captures all
          deformations of K-semistable degenerations of Fano
          threefolds in family~$\textup{\textnumero 3.11}$.

    \item We will prove that every K-semistable degeneration of a Fano
          threefold in family~$\textup{\textnumero 3.11}$ has
          unobstructed deformations; see
          Proposition~\ref{prop:no obstruction}. It follows that
          $\MM^{\textup{K}}_{\textup{\textnumero 3.11}}$ is a smooth
          connected component of $\MM^{\textup{K}}_{3,28}$.
\end{enumerate}
\end{remark}

\bigskip

\section{K-semistable limits of one-parameter families}\label{sec:3}

The aim of this section is to prove the following theorem, which asserts that
every K-semistable degeneration of a smooth Fano threefold of family
\textnumero3.11 arises from the same construction as the smooth members: it is
the blowup of $\bP^3$ along a possibly singular $(2,2)$-complete intersection
curve, followed by the blowup of a fiber of the resulting exceptional divisor. This is the most technical part of the paper, and the point at which the new inputs from \cite{Liu25,KLPZ26} and the theory of Hilbert schemes enter.

\begin{theorem}\label{thm:K-ss-limit}
    Let $X$ be a K-semistable $\bQ$-Fano variety admitting a $\bQ$-Gorenstein deformation
    to the smooth Fano threefolds of family \textup{\textnumero3.11}. Then $X$ is
    isomorphic to the blowup of $Y\coloneqq\Bl_C\bP^3$, for some $(2,2)$-complete
    intersection curve $C\subseteq\bP^3$, along a fiber $\ell_0$ of the exceptional
    divisor over $C$, and $Y$ is generically smooth along $\ell_0$.
\end{theorem}

\begin{rem}
    Equivalently, $Y$ is a hypersurface of bidegree $(1,2)$ in $\bP^1\times\bP^3$ and
    $\ell_0$ is a fiber of the second projection contained in $Y$; this is the form in
    which the theorem will be used in \Cref{sec:Kmod}.
\end{rem}

\subsection{Geometry of anticanonical K3 surfaces and their moduli}

In this subsection, we study the geometry and moduli of anticanonical K3 surfaces on smooth Fano threefolds in family \textnumero 3.11.


Recall from Section~\ref{subsection:3.11} the construction of a smooth member $X$ of family \textnumero 3.11 and the definitions of divisors $H$, $E$ and $F$. Let $S\in |-K_X|$ be a smooth anticanonical K3 surface. Then the restrictions of $H$, $E$ and $F$ to $S$ generate a rank $3$ lattice
\[
\Lambda \coloneqq
\left(
\begin{array}{c|ccc}
& h & e & f \\ \hline \\[-1em]
h & 4 & 4 & 0 \\
e & 4 & 0 & 1 \\
f & 0 & 1 & -2
\end{array}
\right),
\] and $L\coloneqq -K_X|_S=4h-e-2f$ gives a polarization of degree $28$. In fact, one has $\Lambda\simeq U\oplus\langle-28\rangle$. Let $\cF_{(\Lambda,L)}$ be the moduli stack of lattice polarized K3 surfaces (cf. \cite{AE25}), and let $\cF_{28,\Lambda}$ be the Noether--Lefschetz locus in the moduli stack $\cF_{28}$ of degree $28$ polarized ADE K3 surfaces. Let $\cP^{\plt}_{\textup{\textnumero 3.11}}$ be the moduli stack of plt pairs $(X,S)$ (cf. \cite[p.25]{KLPZ26}) consisting of a Gorenstein canonical Fano degeneration of smooth Fano threefolds \textnumero3.11 and an ADE K3 surface $S\in |-K_X|$. By \cite[Corollary 3.18]{KLPZ26}, there exists a proper forgetful morphism $\Phi\colon\cP^{\plt}_{\textup{\textnumero 3.11}}\rightarrow \cF_{28,\Lambda}$ given by $(X,S)\mapsto (S,-K_X|_S)$.

\smallskip

\begin{lem}\label{lem:stab-L}
The group $\Aut(\Lambda,L)\coloneqq\Stab_{O(\Lambda)}(L)$ is trivial. Consequently,
if $S$ is a K3 surface with $\NS(S)\simeq\Lambda$, then $\Aut(S,L)$ is trivial as
well.
\end{lem}

\begin{proof}
Write $x=ah+be+cf$. One computes
\[
(L\cdot h)=12,\quad (L\cdot e)=14,\quad (L\cdot f)=3,\quad
(L\cdot x)=12a+14b+3c,\quad L^2=28 .
\]
Since $L^2>0$, the sublattice $L^{\perp}$ is negative definite by the Hodge index
theorem and $\Lambda\otimes\bQ=\bQ L\oplus(L^{\perp}\otimes\bQ)$; hence
restriction to $L^{\perp}$ gives an injection
$\Aut(\Lambda,L)\hookrightarrow O(L^{\perp})$.

The condition $(L\cdot x)=0$ forces $3\mid b$, so $L^{\perp}$ is spanned by
$h-4f$ and $3e-14f$. Setting
\[
y_1\coloneqq h-4f,\qquad y_2\coloneqq (3e-14f)-4(h-4f)=-4h+3e+2f,
\]
one has $y_1^2=y_2^2=-28$ and $(y_1\cdot y_2)=0$, so
$L^{\perp}\simeq\langle-28\rangle^{\oplus2}$ and $O(L^{\perp})$ is the group of
signed permutations of $\{y_1,y_2\}$, of order $8$. Moreover
\[
h=\tfrac17\big(3L-y_1+y_2\big),\qquad
e=\tfrac12\big(L+y_2\big),\qquad
f=\tfrac1{28}\big(3L-8y_1+y_2\big).
\] Every $x\in\Lambda$ decomposes uniquely as $x=x'+x''$ with $x'\in\bQ L$ and
$x''\in L^{\perp}\otimes\bQ$; since $(x'\cdot L)=(x\cdot L)\in\bZ$ we have
$x'\in\langle L\rangle^{*}$, and likewise $x''\in(L^{\perp})^{*}$. Writing
$\ell\coloneqq L/28$ and $\eta_i\coloneqq y_i/28$, so that
\[
A_{\langle L\rangle}=\langle\ell\rangle\simeq\bZ/28,
\qquad
A_{L^{\perp}}=\langle\eta_1\rangle\oplus\langle\eta_2\rangle\simeq(\bZ/28)^{2},
\]
we obtain a homomorphism
$\Lambda\to A_{\langle L\rangle}\times A_{L^{\perp}}$, $x\mapsto(\bar x',\bar x'')$,
with kernel exactly $\langle L\rangle\oplus L^{\perp}$, hence an embedding of the
glue group
\[
\Theta\coloneqq\Lambda\big/\big(\langle L\rangle\oplus L^{\perp}\big)
\ \hookrightarrow\ A_{\langle L\rangle}\times A_{L^{\perp}} .
\]
Under this embedding the expressions above read
\[
\bar f=\big(3\ell,\ \mu\big),\qquad \mu\coloneqq-8\eta_1+\eta_2,
\qquad \bar h=4\bar f,\qquad \bar e=14\bar f .
\]
Since $h,e,f$ generate $\Lambda$, it follows that $\Theta=\langle\bar f\rangle$
is cyclic.

An isometry $\psi\in O(L^{\perp})$ extends to an isometry of $\Lambda$ fixing $L$
if and only if $(\id\oplus\bar\psi)(\Theta)=\Theta$. Any element of $\Theta$ has
the form $k\bar f=(3k\ell,\,k\mu)$, and $3k\ell=3\ell$ forces $k\equiv1\pmod{28}$,
because $\ell$ has order $28$ and $\gcd(3,28)=1$; hence the condition reduces to
$\bar\psi(\mu)=\mu$ in $A_{L^{\perp}}$. If $\psi(y_i)=\epsilon_iy_i$ with
$\epsilon_i\in\{\pm1\}$, comparing coefficients gives $\epsilon_2\equiv1$ and
$-8\epsilon_1\equiv-8\pmod{28}$; as $1\not\equiv-1$ and $8\not\equiv-8\pmod{28}$,
this forces $\psi=\id$. If instead $\psi(y_1)=\epsilon_1y_2$ and
$\psi(y_2)=\epsilon_2y_1$, comparing $\eta_1$-coefficients forces
$\epsilon_2\equiv-8\equiv20\pmod{28}$ with $\epsilon_2\in\{\pm1\}$, which is
impossible. Hence $\Aut(\Lambda,L)=\{1\}$.

For the second assertion it suffices to show that
$\Aut(S,L)\to O\big(\NS(S),L\big)$ is injective. Let $g$ be in the kernel and let
$g^{*}$ be the induced Hodge isometry of $H^{2}(S,\bZ)$. As $g^{*}$ preserves the
Hodge decomposition, its restriction to $T(S)$ is a Hodge isometry; since
$\rho(S)=3$ is odd, $g^{*}|_{T(S)}=\pm\id$ by \cite[Corollary~3.3.5]{Huy16}. If
$g^{*}|_{T(S)}=\id$, then $g^{*}$ is the identity on $\NS(S)\oplus T(S)$, which
has finite index in the torsion-free group $H^{2}(S,\bZ)$; hence $g^{*}=\id$ and
$g=\id$ by \cite[Proposition 15.2.1]{Huy16}. Suppose instead
$g^{*}|_{T(S)}=-\id$. Since $H^{2}(S,\bZ)$ is unimodular and $\NS(S)$ is
primitive, the quotient $H^{2}(S,\bZ)/\big(\NS(S)\oplus T(S)\big)$ is the graph of
an anti-isometry
$\gamma\colon A_{\NS(S)}\xrightarrow{\ \sim\ }A_{T(S)}$; as $g^{*}$ preserves
$H^{2}(S,\bZ)$ it preserves this graph, which forces $\gamma=-\gamma$, i.e.\
$A_{T(S)}$ is $2$-elementary. This fails, since
$A_{T(S)}\simeq A_{\Lambda}\simeq\bZ/28$. Hence $g=\id$.
\end{proof}

\begin{prop}\label{prop:fiber-P1}
Let $[(S,L)]\in\cF_{28,\Lambda}$ be a very general point. Then the fiber
$\Phi^{-1}([(S,L)])$ is a smooth rational curve.
\end{prop}

\begin{proof}
Since $S$ is very general, $\NS(S)\simeq\Lambda$, and by \Cref{lem:stab-L} the
marking carrying $L$ to $4h-e-2f$ is unique; we may therefore regard $h$, $e$ and
$f$ as well-defined classes on $S$. As $\cP^{\plt}_{\textup{\textnumero}3.11}$ is
irreducible and $S$ is very general, every irreducible component of
$\Phi^{-1}([(S,L)])$ contains at least one pair $(X,S)$ with $X$ smooth; fix such a
pair.

The class $h$ is nef with $h^2=4$ and $(h\cdot f)=0$, so it defines a morphism
$S\to\ove{S}\subseteq\bP^3$ contracting precisely the $(-2)$-curve represented by
$f$; thus $\ove{S}$ is a quartic K3 surface with a single node $p$, the image of
that curve. Since the marking is unique, $h$, and hence the node $p$, are
determined by $[(S,L)]$ alone. The morphism $X\to\bP^3$ restricts on $S$ to this
contraction, so $X$ is obtained by blowing up $\bP^3$ at $p$ and then blowing up
the strict transform of a $(2,2)$-complete intersection curve
$\ove{C}\subseteq\ove{S}$ passing through $p$.

Let $D\in\Lambda$ be the class of the strict transform of $\ove{C}$ on $S$. Since
$\ove{C}$ has degree $4$ and arithmetic genus $1$ and is smooth at $p$,
\[
(D\cdot h)=4,\qquad D^2=0,\qquad (D\cdot f)=1 .
\]
Writing $D=ah+be+cf$, the first and third conditions give $a=-2c$ and $b=1+2c$,
and then $D^2=-14c(c+1)$. Hence $c\in\{0,-1\}$ and
\[
D=e\qquad\text{or}\qquad D=2h-e-f .
\]
The second possibility does not occur. Indeed, the anticanonical class of the
blow-up restricts on $S$ to $4h-D-2f$, which for $D=2h-e-f$ equals $2h+e-f\neq L$,
contradicting $(X,S)\in\Phi^{-1}([(S,L)])$. Therefore $D=e$, and the image
$\ove{e}$ of $e$ in $\NS(\ove{S})$ is the class of $\ove{C}$.

Since $\ove{S}$ is very general among nodal quartic K3 surfaces containing a
$(2,2)$-complete intersection curve through the node, $|\ove{e}|$ is a pencil, and
every member passes through $p$, is smooth at $p$, and has at worst nodal
singularities. Let $\ove{C}_t$, $t\in\bP^1$, denote its members. Blowing up
$\Bl_p\bP^3$ along the strict transform of $\ove{C}_t$ produces a Fano threefold
$X_t$ with at worst nodal singularities; in particular $X_t$ is a Gorenstein
terminal degeneration of the smooth members of family \textnumero3.11 containing
$(S,L)$ as an anticanonical K3 surface. Thus $t\mapsto(X_t,S)$ defines a finite
morphism $\bP^1\to\Phi^{-1}([(S,L)])$. By \cite[Theorem~3.4]{KLPZ26} the fiber is
smooth of local dimension $1$ at each $(X_t,S)$, so the image is a connected
component isomorphic to $\bP^1$. By the previous paragraph every component of
$\Phi^{-1}([(S,L)])$ arises this way, so this component is the whole fiber.
\end{proof}

\begin{rem}\label{rem:involution}
    For the K3 surface $S$ in \Cref{prop:fiber-P1}, the classes $e$ and $2h-e-f$ play
    symmetric roles. Indeed, let $\ove{S}\subseteq\bP^3$ be a quartic K3 surface with a
    node $p$, and let $\ove{C}\subseteq\ove{S}$ be a $(2,2)$-complete intersection curve
    through $p$. Any quadric $Q$ containing $\ove{C}$ cuts out
    $Q\cap\ove{S}=\ove{C}+\ove{C}'$, and the residual curve $\ove{C}'$ is again a
    $(2,2)$-complete intersection through $p$. Since $\ove{C}+\ove{C}'\sim 2h$ and
    $\ove{C}'$ passes through the node, the proper transform of $\ove{C}'$ on $S$ has
    class $2h-e-f$.

    The isometry $\iota\in O(\Lambda)$ fixing $h$ and $f$ and exchanging
    $e\leftrightarrow 2h-e-f$ carries the polarization $L=4h-e-2f$ to
    \[
    4h-(2h-e-f)-2f\ =\ 2h+e-f,
    \]
    which is again a degree $28$ polarization on $S$. Hence $\iota$ induces a generically
    non-trivial involution on $\cF_{28,\Lambda}$, interchanging the two pencils $|e|$ and
    $|2h-e-f|$ of $(2,2)$-complete intersection curves through the node.
\end{rem}

\subsection{Hilbert scheme of curves in $\bP^3$}

In this subsection, we classify the flat degenerations in $\bP^3$ of families of
$(2,2)$-complete intersections.

\begin{prop}\label{lem:Hilbert-scheme}
    Let $C\subseteq\bP^3$ be a flat degeneration of $(2,2)$-complete intersections.
    Then, in suitable coordinates, $C$ is one of the following, where $q_1,q_2$ are
    quadrics, $\ell$ a linear form and $c$ a constant:
    \begin{enumerate}
        \item[\textup{(1)}] a $(2,2)$-complete intersection;
        \item[\textup{(2)}] a plane cubic union an incident line,
              $I_C=(xy,\,xz,\,yq_1+zq_2)$;
        \item[\textup{(3)}] an arithmetically Cohen--Macaulay \textup{(abbr.} ACM\textup{)} double-plane limit,
              $I_C=(x^2,\,xz,\,xq_1+zq_2)$;
        \item[\textup{(4)}] a plane quartic with an embedded doublet,
              $I_C=(x^2,\,xy,\,xq_1,\,y^2q_2+yq_1\ell+cq_1^2)$.
    \end{enumerate}
\end{prop}

\begin{proof}
By \cite[Theorem~5.2]{AV92}, the component of the Hilbert scheme $\Hilb_{4t}\bP^3$ containing the
$(2,2)$-complete intersections is stratified into four loci: the locus of complete intersections and the three boundary strata parametrizing,
 the curves described in (2),(3) and (4) respectively. The normal forms of their ideals given in \emph{loc.\ cit.} are exactly those displayed above, up to a projective change of coordinates.
\end{proof}

\begin{rem}
    Case (3) arises as a limit of case (2), obtained by letting the line approach the
    plane of the cubic: substituting $x+ty$ for $y$ and letting $t\to 0$ carries
    $I_C=(xy,\,xz,\,yq_1+zq_2)$ to $I_C=(x^2,\,xz,\,xq_1+zq_2)$.
\end{rem}

\subsection{Construction of K-unstable degenerations}

In this subsection we study the K3 surfaces parametrized by the two Noether--Lefschetz
divisors $\cD_{\Lambda_{\textup{I}}}$ and $\cD_{\Lambda_{\textup{II}}}$ in
$\cF_{28,\Lambda}$, together with the Gorenstein canonical Fano threefolds containing
them as anticanonical divisors. Showing that both loci consist of K-unstable members is
the key new ingredient in the proof of \Cref{prop:complete-intersection}, and hence of
\Cref{thm:K-ss-limit}.

\subsubsection{Construction I: K-unstable members of family \textnumero3.11}\label{subsec:K-unstable-I}

Let $S\subset\bP^3$ be a very general quartic K3 surface containing two skew lines
$L_1$ and $L_2$ and having a single $A_1$-singularity at a point $p$ on one of these two lines, say $L_2$.
Let $\sigma\colon\wt S\to S$ be the minimal resolution, with exceptional
$(-2)$-curve $F$, and let $\wt L_2$ denote the proper transform of $L_2$. Set
\[
H\coloneqq \sigma^*\big(\cO_{\bP^3}(1)|_S\big),\qquad
\Gamma\coloneqq \sigma^{-1}_*L_1,\qquad
E\coloneqq (H-\Gamma)+\wt L_2 .
\]
Since $p\notin L_1$, the resolution is an isomorphism near $L_1$, so
$\Gamma^2=-2$ and $(\Gamma\cdot F)=0$; moreover $(H\cdot\Gamma)=(H\cdot\wt L_2)=1$,
$(\wt L_2\cdot F)=1$, and $(\wt L_2\cdot\Gamma)=0$ because $L_1\cap L_2=\emptyset$.
A direct computation therefore shows that the sublattice
$\Lambda_{\textup{I}}\coloneqq\langle H,E,F,\Gamma\rangle\subseteq\NS(\wt S)$
has Gram matrix
\begin{equation}\label{eq:Lambda1}
\Lambda_{\textup{I}}\ =\ \left(
\begin{array}{c|cccc}
& H  & E & F & \Gamma \\ \hline \\[-1em]
H & 4 & 4 & 0 & 1 \\
E & 4 & 0 & 1 & 3 \\
F & 0 & 1 & -2 & 0 \\
\Gamma & 1 & 3 & 0 & -2
\end{array}\right),
\qquad \det \Lambda_{\textup{I}}=-31,
\end{equation}
the inclusion being an equality for very general such $S$. Moreover $\wt S$ is
quasi-polarized by $4H-E-2F$, which satisfies $(4H-E-2F)^2=28$.

Let $\{H_t\}_{t\in\bP^1}$ be the pencil of planes through $L_1$, and let $G_t$
denote the residual curve to $L_1$ in $\sigma^*(H_t\cap S)$, so that
$G_t\in|H-\Gamma|$. If $p\notin H_t$, then $G_t$ is a plane cubic. For the unique
$t_0$ with $p\in H_{t_0}$, the plane $H_{t_0}$ meets $L_2$ at $p$, which forces
$\mult_p(H_{t_0}\cap S)=2$ along the residual curve; writing
$\wt G\coloneqq (H-\Gamma)-F$ for the proper transform, one computes
$\wt G^2=-2$ and $(\wt G\cdot F)=2$. Hence $G_{t_0}=\wt G+F$ is a
\emph{banana curve}, that is, a union of two smooth rational curves meeting
nodally at two distinct points. In either case,
\[
p_a(G_t)=1
\qquad\text{and}\qquad
(G_t\cdot\wt L_2)=1 .
\]

Let $\wt{\bP}^3\coloneqq \Bl_p\bP^3$, and let $X_t\to\wt{\bP}^3$ be the blow-up
along $E_t\coloneqq G_t\cup\wt L_2$, a curve of class $E$ and arithmetic genus $1$.
For $t\neq t_0$ the point $G_t\cap\wt L_2$ is a smooth point of $G_t$, so $E_t$ is
nodal with a single node; for $t=t_0$ one has $E_{t_0}=\wt G\cup F\cup\wt L_2$ with
$(\wt G\cdot F)=2$ and $(\wt L_2\cdot F)=1$, so $E_{t_0}$ is nodal with three nodes.
In particular $E_t$ is lci for every $t$, and $X_t$ is a Gorenstein terminal Fano
threefold degenerating the members of family \textnumero3.11, its singularities
being exactly one ordinary double point over each node of $E_t$. By
\Cref{prop:K-unstable}, $X_t$ is K-unstable for general $t$; since the K-unstable
locus is closed, $X_t$ is K-unstable for every $t\in\bP^1$.

Finally, the moduli stack of lattice-polarized K3 surfaces is smooth, and the
forgetful morphism $\Phi\colon\cP^{\plt}_{\textup{\textnumero}3.11}\to\cF_{28,\Lambda}$
is proper with connected generic fiber; by Zariski's main theorem, every fiber of
$\Phi$ is therefore connected. For each $t\in\bP^1$ the pair $(X_t,S)$ lies in
$\Phi^{-1}([S])$ and $X_t$ is Gorenstein terminal, so by
\cite[Corollary~3.23]{KLPZ26} the fiber $\Phi^{-1}([S])$ is smooth at
$[(X_t,S)]$ of dimension $h^{1,2}(X_t)=1$. Thus the image of
$\bP^1\to\Phi^{-1}([S])$, $t\mapsto[(X_t,S)]$, is open in $\Phi^{-1}([S])$; being
proper it is also closed, and connectedness of the fiber forces it to be all of
$\Phi^{-1}([S])$. Consequently every $(X,S)\in\Phi^{-1}([S])$ satisfies
$X\simeq X_t$ for some $t\in\bP^1$, and in particular $X$ is K-unstable.
\begin{prop}\label{prop:K-unstable-I}
Let $(X,S)\in\cP^{\plt}_{\textup{\textnumero 3.11}}$. If
$(S,-K_X|_S)$ lies in the Noether--Lefschetz divisor
$\cD_{\Lambda_{\textup{I}}}$, then $X$ is K-unstable.
\end{prop}

\begin{proof}
By the discussion above, over a very general point of
$\cD_{\Lambda_{\textup{I}}}$, the fiber of $\Phi$ consists entirely of
pairs whose threefold component is K-unstable. Let $\cZ_0$ be the unique
irreducible component of
$\Phi^{-1}(\cD_{\Lambda_{\textup{I}}})$ that dominates
$\cD_{\Lambda_{\textup{I}}}$. If $\cZ_0$ contained a K-semistable pair,
then the openness of K-semistability would imply that a general point of
$\cZ_0$ parametrizes a K-semistable threefold, contradicting the preceding
observation. Hence every pair parametrized by $\cZ_0$ has K-unstable
threefold component.

It remains to exclude K-semistable pairs lying on components that do not
dominate $\cD_{\Lambda_{\textup{I}}}$. Suppose, to the contrary, that
$(X,S)$ is such a pair, and let $\cZ$ be an irreducible component of
$\Phi^{-1}(\cD_{\Lambda_{\textup{I}}})$ containing $(X,S)$. Although
$\cD_{\Lambda_{\textup{I}}}$ need not be $\bQ$-Cartier, its pullback to
the smooth normalization of $\cF_{28,\Lambda}$, namely the moduli stack of
$\Lambda_{\textup{I}}$-polarized K3 surfaces, is Cartier. Hence, after
base change to the normalization, its inverse image under $\Phi$ is a
Cartier divisor and is therefore pure of codimension one. Since the
normalization is finite and surjective, it follows that
$\Phi^{-1}(\cD_{\Lambda_{\textup{I}}})$ is purely divisorial. Since
$\cZ$ does not dominate $\cD_{\Lambda_{\textup{I}}}$, a dimension count
shows that the fibers of $\Phi|_{\cZ}$ have dimension at least two.

If $X$ were terminal, then
\cite[Theorem~3.4 and Corollary~3.8]{KLPZ26} would imply that $\Phi$ is
smooth at $(X,S)$ of relative dimension $h^{1,2}(X)=1$, a contradiction.
Thus $X$ is not terminal. By the openness of K-semistability, we may
replace $(X,S)$ by a general point of a K-semistable open subset of
$\cZ$. After shrinking further to an equisingular stratum, the
one-dimensional singular loci form a flat family, and blowing them up
gives a simultaneous terminalization. We thereby obtain a family of
Gorenstein terminal weak Fano threefolds of dimension at least two over
a fixed polarized K3 surface. Since the original family is recovered by
taking anticanonical models, the terminalized family still has dimension
at least two. It follows from \cite[Theorem~3.4]{KLPZ26} that
$h^{1,2}(\widetilde X)\geq 2$ for the terminalization
$\widetilde X$ in this family.

Following the argument in the proof of
\cite[Theorem~3.10]{KLPZ26}, the weak Fano threefold $\widetilde X$
deforms to a Gorenstein terminal weak Fano threefold whose anticanonical
model $X'$ is terminal. Moreover,
\cite[Remark~3.12]{KLPZ26} gives
$h^{1,2}(X')\geq h^{1,2}(\widetilde X)\geq 2$. The threefold $X'$ admits
a $\bQ$-Gorenstein smoothing to a smooth Fano threefold of anticanonical
volume $28$, and the invariance of $h^{1,2}$ in
\cite[Corollary~3.8]{KLPZ26} shows that its smooth fiber also satisfies
$h^{1,2}\geq 2$.

On the other hand, the Iskovskikh--Mori--Mukai classification shows that
the smooth Fano threefolds of anticanonical volume $28$ belong to the
families
\textup{\textnumero 2.21},
\textup{\textnumero 3.11},
\textup{\textnumero 3.12},
\textup{\textnumero 4.2}, and
\textup{\textnumero 5.1},
and a smooth member of each of these families has $h^{1,2}\leq 1$.
This contradiction completes the proof.
\end{proof}

\subsubsection{Construction II of K-unstable Fano}

The results proved in this subsection are not needed later. We include them because they
describe the behavior along the second Noether--Lefschetz divisor
$\cD_{\Lambda_{\textup{II}}}$, and because the argument illustrates how the deformation
theory of \cite{KLPZ26} is applied when the relevant degeneration fails to be terminal, a
situation not covered by \cite[Corollary~3.23]{KLPZ26}.

Let $S\subset\bP^3$ be a very general quartic K3 surface containing two skew lines
$L_1$, $L_2$ and having a single $A_1$-singularity at a point $p\in L_1$, and let
$\sigma\colon\wt S\to S$ be the minimal resolution, with exceptional $(-2)$-curve $F$.
Let $\{H_t\}_{t\in\bP^1}$ be the pencil of planes through $L_1$, let $C_t$ be the
residual curve to $L_1$ in the hyperplane section $H_t\cap S$, and let $\wt C_t$ be its
proper transform on $\wt S$. Since every $H_t$ contains $p$, each $C_t$ is a plane cubic
passing through $p$; since $L_1\cap L_2=\emptyset$, the curve $\wt C_t$ meets $L_2$
nodally at one point. Set
\[
H\coloneqq \sigma^*\big(\cO_{\bP^3}(1)|_S\big),\qquad
\Gamma\coloneqq \wt L_1,\qquad
E_t\coloneqq \wt C_t+L_2,
\]
where $\wt L_1$ denotes the proper transform of $L_1$. A direct computation then gives
\[
\NS(\wt S)\ \supseteq\
\left(
\begin{array}{c|cccc}
& H  & E_t & F & \Gamma \\ \hline \\[-1em]
H & 4 & 4 & 0 & 1 \\
E_t & 4 & 0 & 1 & 2\\
F & 0 & 1 & -2 & 1 \\
\Gamma & 1 & 2 & 1 & -2
\end{array}\right),
\qquad \det=-31,
\]
with equality for very general such $S$, and $\wt S$ is quasi-polarized by
$4H-E_t-2F$.

Let $\wt{\bP}^3\coloneqq \Bl_p\bP^3$ and let $X_t\to\wt{\bP}^3$ be the blowup along
$E_t$. Then $X_t$ is a Gorenstein terminal weak Fano threefold whose anticanonical ample
model $\ove{X}_t$ is a Gorenstein canonical degeneration of Fano threefolds of family
\textnumero3.11. The next result shows that $\ove{X}_t$ is K-unstable.

\begin{prop}\label{prop:K-unstable-2}
    Let $C\coloneqq C_1\cup_p C_2\subseteq \bP^3$ be the union of a plane cubic $C_1$ and
    a line $C_2$ meeting nodally at a point $p$, and let $q\in C_1$ be a smooth point.
    Set $\wt{\bP}^3\coloneqq \Bl_q\bP^3$, let $\wt{C}$ be the proper transform of $C$,
    and let $X\coloneqq \Bl_{\wt{C}}\wt{\bP}^3$. Then $X$ is a weak Fano variety whose
    anticanonical ample model $\ove{X}$ is K-unstable.
\end{prop}

\begin{proof}
    Let $\wh{X}\coloneqq \Bl_{C_2}\wt{\bP}^3$, which is the unique smooth Fano threefold
    of family \textnumero3.26, and let $\wt{X}$ be the blow-up of $\wh{X}$ along the
    proper transform of $C_1$. One then has the following commutative diagram:
    \[\begin{tikzcd}[ampersand replacement=\&]
	{\wt{X}} \& {\wh{X}} \& {\wt{\bP}^3} \& {\bP^3} \\
	\& X \& {\ove{X}}
	\arrow["\psi", from=1-1, to=1-2]
	\arrow["\phi"{description}, from=1-1, to=2-2]
	\arrow["\tau", from=1-2, to=1-3]
	\arrow[from=1-3, to=1-4]
	\arrow[from=2-2, to=1-3]
	\arrow["\pi", from=2-2, to=2-3]
    \end{tikzcd}.\]
    The morphism $\phi\colon\wt{X}\rightarrow X$ is a small resolution contracting a
    smooth rational curve. Let $\wt{E}_1$ (resp.\ $\wt{E}_2$, $\wt{F}$) be the
    exceptional divisor of $\wt{X}\to \bP^3$ over $C_1$ (resp.\ $C_2$, $q$), and let
    $\wt{H}$ be the pullback of $\cO_{\bP^3}(1)$. Then
    \[
    -K_{\wt{X}}\ =\ 4\wt{H}-\wt{E}_1-\wt{E}_2-2\wt{F}.
    \]
    Let $\Pi\subseteq\bP^3$ be the unique plane spanned by $C_1$, and let
    $\wt{\Pi}\subseteq \wt{X}$ be its proper transform. Then $\wt{\Pi}$ is
    isomorphic to the smooth del Pezzo surface of degree $7$, and
    $\wt{\Pi}\to \Pi\simeq \bP^2$ is the blow-up at the two points $p,q$. Denote
    by $L\in \Pic(\wt{\Pi})$ the pullback of $\cO_{\bP^2}(1)$, and by $e_1,e_2$ the
    exceptional divisors over $p,q$ respectively. Since $\Pi$ contains $C_1$ and is
    smooth at $q$, we have $\wt{\Pi}\sim \wt{H}-\wt{E}_1-\wt{F}$, and hence
    \[
    -K_{\wt{X}}\ =\ \wt{H}+(\wt{H}-\wt{E}_2)+(\wt{H}-\wt{F})+\wt{\Pi},
    \]
    where $\wt{H}$, $\wt{H}-\wt{E}_2$ and $\wt{H}-\wt{F}$ are all nef.

    Let $\Gamma$ be an integral curve on $\wt{X}$ with
    $(-K_{\wt{X}}\cdot \Gamma)\leq 0$. If $(-K_{\wt{X}}\cdot\Gamma)<0$, then the
    displayed decomposition forces $(\wt{\Pi}\cdot\Gamma)<0$, so
    $\Gamma\subseteq \wt{\Pi}$; however,
    \[
    -K_{\wt{X}}\big|_{\wt{\Pi}} \ \sim \ 4L-(3L-e_1-e_2)-e_1-2e_2 \ = \ L-e_2
    \]
    is nef, a contradiction. Hence $(-K_{\wt{X}}\cdot\Gamma)=0$. If
    $\Gamma\subseteq \wt{\Pi}$, then $\big((L-e_2)\cdot\Gamma\big)=0$, so $\Gamma$ is
    either $e_1$ or the proper transform of a line in $\Pi$ through $q$; that is,
    $\Gamma$ is a component of a fiber of the morphism $\wt{\Pi}\to\bP^1$ defined by
    $|L-e_2|$. If $\Gamma\not\subseteq \wt{\Pi}$, then each summand above vanishes on
    $\Gamma$, so
    \[
    \big(\Gamma\cdot\wt{\Pi}\big) \ =\  \big(\Gamma\cdot(\wt{H}-\wt{E}_2)\big) \ =\
    \big(\Gamma\cdot(\wt{H}-\wt{F})\big) \ =\  \big(\Gamma\cdot\wt{H}\big) \ =\  0.
    \]
    The only curves on $\wt{X}$ not contained in $\wt{\Pi}$ meeting both
    $\wt{H}-\wt{E}_2$ and $\wt{H}-\wt{F}$ trivially are the proper transforms of lines in
    $\bP^3$ through $q$ meeting $C_2$; but these have intersection number $1$ with
    $\wt{H}$, so no such $\Gamma$ exists.

    Therefore $-K_{\wt{X}}$ and $-K_X$ are both nef, so $\wt{X}$ and $X$ are weak Fano
    varieties. The morphism $\wt{X}\rightarrow \ove{X}$ contracts $\wt{\Pi}$ onto the
    base $\bP^1$ of $|L-e_2|$, producing $D_{\infty}$-singularities along a rational
    curve of degree $1$ with respect to $-K_{\ove{X}}$. By either
    \cite[Theorem 1.3(1)]{Liu25} or \cite[Theorem 3.11]{KLPZ26}, $\ove{X}$ is K-unstable.
\end{proof}

\begin{rem}\label{rem:Picard-rank}
    Since $\wt{X}$ is smooth, the construction of $\wt{X}$ gives $\rho(\wt{X})=4$. The
    small contraction $\phi\colon \wt{X}\to X$ contracts precisely the curve class $e_1$,
    while the divisorial contraction $\pi\colon X\to\ove{X}$ contracts precisely the
    curve class $L-e_2$. Consequently, $\rho(X)=3$ and $\rho(\ove{X})=2$.
\end{rem}

Unlike the case of $\Lambda_{\textup{I}}$, here $\ove{X}_t$ is not terminal, so we
cannot apply \cite[Corollary~3.23]{KLPZ26} directly to conclude that every
$(X,S)\in\Phi^{-1}\big([(S,L)]\big)$ satisfies $X\simeq\ove{X}_t$. We argue by
contradiction. Since $\Phi$ has connected fibers, if some
$(X,S)\in\Phi^{-1}\big([(S,L)]\big)$ had $X\not\simeq\ove{X}_t$ for all $t$, we could
choose $(X,S)$ so that, writing $(X_0,S)\coloneqq(X,S)$, it admits a deformation to a
pair $(X_b,S)\in\Phi^{-1}\big([(S,L)]\big)$ with $X_b\not\simeq\ove{X}_t$ for any $t$.
Let $\sX\to B$ be the corresponding degeneration of $X_b$ to $\ove{X}$.

By \Cref{rem:Picard-rank}, $\ove{X}$ has Picard rank $2$, hence $\rho(X_b)\geq 2$ by
\cite[Lemma 2.9]{KLPZ26}; on the other hand, $X_b$ is a degeneration of smooth Fano
threefolds $X_{\eta}$ of family \textnumero3.11, so the same lemma gives
$\rho(X_b)\leq 3$. Consider the degeneration from a smooth Fano threefold of family
\textnumero3.11 to $X_b$ and then to $X=X_0$, which can be realized over a surface. A
$\bQ$-factorialization of the total space yields a degeneration from a smooth Fano
threefold of family \textnumero3.11 to a weak Fano variety $X'_b$ with ample model
$X_b$, and then to a weak Fano variety $X'$ with ample model $\ove{X}$.
\begin{lem}\label{lem:uniqueness-of-model}
    The two weak Fano varieties $X$ and $X'$ are isomorphic.
\end{lem}

\begin{proof}
    Let $\theta\colon X'\to\ove{X}$ be the birational morphism to the
    anticanonical model. We first show that $\theta$ is not small.

    Let $\sX\to(0\in C)$ and $\sX'\to(0\in C)$ be two birational families
    of weak Fano varieties degenerating smooth Fano threefolds of family
    \textnumero3.11 to weak Fano models of $\ove{X}$, with the same relative
    anticanonical model $\ove{\sX}\to C$. The former is obtained from the
    blowup construction and satisfies $\sX_0=X$, while the latter is obtained
    by $\bQ$-factorialization and satisfies $\sX'_0=X'$. Since both $\sX$
    and $\sX'$ are $\bQ$-factorial terminal weak Fano varieties over $C$,
    they are isomorphic in codimension one and are connected by a sequence
    of flops over $\ove{\sX}$.

    By \Cref{prop:K-unstable-2,rem:Picard-rank}, the exceptional locus of
    $\sX\to\ove{\sX}$ contains a surface $\sR$ dominating the rational curve
    $\Gamma\subset\ove{X}$. Since $-K$ is ample on every non-central fiber, all the flopping loci are contained in the central fiber. The exceptional loci of each flop on both sides are surfaces dominating $\Gamma$. Tracking this exceptional component through the sequence of flops, we obtain a surface in the central fiber of $\sX'$ dominating $\Gamma$. It follows that
    \[
        \theta\colon X'\longrightarrow\ove{X}
    \]
    has an exceptional divisor dominating $\Gamma$. In particular, $\theta$
    is not small.

    Let $D_1,\dots,D_k$ be the $\theta$-exceptional prime divisors, so that
    $k\geq1$. Since $\theta$ is crepant, each $D_i$ has discrepancy zero over
    $\ove{X}$. As every singularity of $\ove{X}$ is a cDV hypersurface
    singularity, no crepant divisor over $\ove{X}$ has center a closed point
    by \cite[Theorem 5.34]{KM98}. Hence every $D_i$ dominates $\Gamma$. Let $\xi$ be the generic point of $\Gamma$. After localizing at $\xi$,
    the germ of $\ove{X}$ is an $A_1$ surface singularity, which admits a
    unique crepant divisorial valuation, namely that defined by the
    exceptional curve of its minimal resolution. Thus $k=1$. Write
    $R'\coloneqq D_1$. By the same uniqueness, $R\coloneqq\wt{\Pi}$ and $R'$
    define the same valuation over $\ove{X}$. Therefore $X$ and $X'$ are
    isomorphic in codimension one over $\ove{X}$.

    Finally, since $H$ and $H'$ are relatively ample over $\ove{X}$ and
    both $X$ and $X'$ are normal, one has
    \[
        X
        \ \simeq\
        \Proj_{\ove{X}}\bigoplus_{m\geq0}\pi_*\cO_X(mH)
        \ \simeq\
        \Proj_{\ove{X}}\bigoplus_{m\geq0}\theta_*\cO_{X'}(mH')
        \ \simeq\
        X'.\qedhere
    \]
\end{proof}

Let $\wt{S}\subseteq X$ be the proper transform of $S$, which is again an anticanonical
K3 surface. Then $(X,\wt{S})$ admits a $\bQ$-Gorenstein deformation to $(X'_b,S)$
lifting the deformation from $(\ove{X},S)$ to $(X_b,S)$. By
\cite[Theorem~3.4]{KLPZ26}, the morphism $\Def(X,\wt{S})\to\Def(\wt{S})$ is smooth onto
its image, of relative dimension $1$. But $X$ deforms both to $X_t$ for $t\in\bP^1$ and
to $X_b$ for $b\in B$, so the fiber of $\Def(X,\wt{S})\to\Def(\wt{S})$ through $X$ has
two distinct branches at $X$, hence is locally a nodal curve, contradicting smoothness.
Therefore every $(X,S)\in\Phi^{-1}\big([(S,-K_X|_S)]\big)$ satisfies $X\simeq X_t$ for
some $t\in\bP^1$, and the same argument as in \Cref{prop:K-unstable-I} yields the
following.

\begin{prop}\label{prop:K-unstable-II}
    If $(X,S)\in\cP^{\plt}_{\textup{\textnumero}3.11}$ is such that $(S,-K_X|_S)$ lies in
    the Noether--Lefschetz divisor $\cD_{\Lambda_{\textup{II}}}$, then $X$ is K-unstable.
\end{prop}

\subsection{K-semistable limits of one-parameter families}

Let $X$ be a K-semistable $\bQ$-Fano variety admitting a $\bQ$-Gorenstein
smoothing $\pi\colon\sX\to T$ over a smooth pointed curve $0\in T$ such that
$\sX_0\simeq X$ and such that every fiber $\sX_t$ with $t\in T\setminus\{0\}$ is
a smooth Fano threefold in the family \textnumero3.11. Set
$T^{\circ}\coloneqq T\setminus\{0\}$ and $\sX^{\circ}\coloneqq\sX\times_T T^{\circ}$.
After a finite base change we may assume that there are
\begin{itemize}
  \item a section $\sP^{\circ}\subseteq\bP^3_{T^{\circ}}\to T^{\circ}$, and
  \item a smooth family $\sC^{\circ}\subseteq\bP^3_{T^{\circ}}$ of $(2,2)$-complete
        intersection curves containing $\sP^{\circ}$,
\end{itemize}
such that
\[
  \sX^{\circ}\ \simeq\ \Bl_{\wt{\sC}^{\circ}}\bigl(\Bl_{\sP^{\circ}}\bP^3_{T^{\circ}}\bigr),
\]
where $\wt{\sC}^{\circ}$ denotes the proper transform of $\sC^{\circ}$ in
$\Bl_{\sP^{\circ}}\bP^3_{T^{\circ}}$. Let $\sE^{\circ}$ and $\sF^{\circ}$ be the
exceptional divisors of the induced morphism $\sX^{\circ}\to\bP^3_{T^{\circ}}$
lying over $\sC^{\circ}$ and $\sP^{\circ}$ respectively, and let $\cH^{\circ}$ be
an effective divisor on $\sX^{\circ}$ linearly equivalent to the pullback of
$\cO_{\bP^3_{T^{\circ}}}(1)$. Taking Zariski closures, we obtain Weil divisors
$\sE$, $\sF$ and $\cH$ on $\sX$ extending $\sE^{\circ}$, $\sF^{\circ}$ and
$\cH^{\circ}$, and we denote by $E$, $F$ and $H$ their respective restrictions to
the central fiber $X=\sX_0$.

Let $f\colon\sZ\to\sX$ be a $\bQ$-factorialization. Then $\sZ$ is klt and
$-K_{\sZ}=-f^{*}K_{\sX}$ is trivial over $\sX$, so $\sZ$ is weak Fano, and in
particular a Mori dream space, over $\sX$. Running an $f^{-1}_{*}\cH$-MMP for
$\sZ$ over $\sX$ therefore terminates with a log minimal model
$\sZ\dashrightarrow\wt{\sX}$, fitting into a commutative diagram
\[
\begin{tikzcd}[ampersand replacement=\&]
	{\wt{\sX}} \&\& \sX \\
	\& T
	\arrow["{g}", from=1-1, to=1-3]
	\arrow["{\wt{\pi}}"', from=1-1, to=2-2]
	\arrow["\pi", from=1-3, to=2-2]
\end{tikzcd}
\]
with the following properties.
\begin{enumerate}
  \item $g$ is a small birational morphism, and
        $g^{\circ}\colon\wt{\sX}|_{T^{\circ}}\to\sX|_{T^{\circ}}$ is an isomorphism.
  \item $-K_{\wt{\sX}}=g^{*}(-K_{\sX})$ is a Cartier divisor which is big and nef
        relatively over $T$.
  \item $\wt{X}\coloneqq\wt{\sX}_0$ is a Gorenstein canonical
        weak Fano variety (cf. \cite[Theorem 4.2(3)]{LZ25}) whose anticanonical model is isomorphic to $X=\sX_0$.
  \item The birational transform $\wt{\cH}$ of $\cH$ is nef relatively over $\sX$.
\end{enumerate}
By (1) we may regard $\sE^{\circ}$, $\sF^{\circ}$ and $\cH^{\circ}$ as divisors on
$\wt{\sX}|_{T^{\circ}}$; let $\wt{\sE}$, $\wt{\sF}$ and $\wt{\cH}$ denote their
closures in $\wt{\sX}$, which are $\bQ$-Cartier since $\wt{\sX}$ is $\bQ$-factorial. Then
$\wt{E}\coloneqq\wt{\sE}|_{\wt{X}}$, $\wt{F}\coloneqq\wt{\sF}|_{\wt{X}}$ and
$\wt{H}\coloneqq\wt{\cH}|_{\wt{X}}$ are $\bQ$-Cartier, hence Cartier by
\cite[Theorem 4.2(3)]{LZ25}. It then follows from \cite[Theorem A.1]{HLS26} that
$\wt{\sE}$, $\wt{\sF}$ and $\wt{\cH}$ are themselves Cartier. Finally, one has
\[
  -K_{\wt{\sX}}\ \sim\ 4\wt{\cH}-\wt{\sE}-2\wt{\sF}.
\] 

\begin{prop}\label{prop:nefness-on-K3}
  Let $\wt{S}\in|-K_{\wt{X}}|$ be a general member. Then $\wt{H}|_{\wt{S}}$ is nef.
\end{prop}

\begin{proof}
  Since $\wt{X}$ is a Gorenstein canonical weak Fano threefold, the surface $\wt{S}$
  is a K3 surface with ADE singularities (cf.\ \cite{Rei83,Sho79}). Let
  $\mu\colon\wh{S}\to\wt{S}$ be its minimal resolution, so that $\wh{S}$ is a smooth
  K3 surface, and set
  \[
    \wh{H}\coloneqq\mu^{*}\bigl(\wt{H}|_{\wt{S}}\bigr),\qquad
    \wh{E}\coloneqq\mu^{*}\bigl(\wt{E}|_{\wt{S}}\bigr),\qquad
    \wh{F}\coloneqq\mu^{*}\bigl(\wt{F}|_{\wt{S}}\bigr).
  \]
  As $\wt{H}|_{\wt{S}}$ is nef if and only if $\wh{H}$ is, it suffices to prove the
  latter.

  Assume $\wh{H}$ is not nef. Then some $(-2)$-class in $\NS(\wh{S})$, represented by
  a smooth rational curve $\Gamma\subseteq\wh{S}$, satisfies
  \[
    a\ \coloneqq\ (\wh{H}\cdot\Gamma)\ =\ \bigl(\wt{H}\cdot\mu(\Gamma)\bigr)\ <\ 0 .
  \]
  Because $\wt{\cH}$ is relatively nef over $\sX$, the curve $\mu(\Gamma)$ is not
  contracted by $g\colon\wt{X}\to X$, whence
  \[
    \bigl((4\wh{H}-\wh{E}-2\wh{F})\cdot\Gamma\bigr)
    \ =\ \bigl(-K_{\wt{X}}\cdot\mu(\Gamma)\bigr)
    \ =\ \bigl(-g^{*}K_{X}\cdot\mu(\Gamma)\bigr)
    \ =\ \bigl(-K_{X}\cdot (g\circ\mu)(\Gamma)\bigr)
    \ >\ 0 .
  \]
  Writing $b\coloneqq(\wh{E}\cdot\Gamma)$ and $c\coloneqq(\wh{F}\cdot\Gamma)$, we
  therefore have the Gram matrix
  \begin{equation*}
    \left(
    \begin{array}{c|cccc}
      & \wh{H} & \wh{E} & \wh{F} & \Gamma \\ \hline \\[-1em]
      \wh{H} & 4 & 4 & 0 & a \\
      \wh{E} & 4 & 0 & 1 & b \\
      \wh{F} & 0 & 1 & -2 & c \\
      \Gamma & a & b & c & -2
    \end{array}\right)
  \end{equation*}
  on $\wh{S}$, whose determinant
  \[
    \Delta\ =\ a^{2}-16ab-8ac+8b^{2}+8bc+16c^{2}-56
  \]
  satisfies $\Delta\le0$ by the Hodge index theorem. Set $d\coloneqq 4a-b-2c$, a positive integer by the displayed inequality
  above. Substituting $b=4a-2c-d$ and completing the square in $c$ gives
  \[
    \Delta\ =\ 32\Bigl(c+\tfrac{3(d-3a)}{8}\Bigr)^{2}
              +\tfrac{7}{2}\bigl(7a^{2}-6ad+d^{2}\bigr)-56 ,
  \]
  so that $\Delta\le0$ forces
  \[
    7a^{2}-6ad+d^{2}\ \le\ 16 .
  \]
  Since $a\le-1$ and $d\ge1$, we have $-6ad\ge6$ and $d^{2}\ge1$,
  hence $7a^{2}\le9$ and therefore $a=-1$. The inequality then reads
  $d^{2}+6d\le9$, which forces $d=1$. For these values,
  \[
    \Delta\ =\ 32c^{2}+96c+65\ =\ 32\Bigl(c+\tfrac{3}{2}\Bigr)^{2}-7\ \ge\ 1
  \]
  for every integer $c$, contradicting $\Delta\le0$. Hence $\wh{H}$, and therefore
  $\wt{H}|_{\wt{S}}$, is nef.
\end{proof}

\begin{prop}\label{prop:big and nef on X}
  The divisor $\wt{H}$ is nef and big on $\wt{X}$.
\end{prop}

\begin{proof}
  Once nefness is known, bigness follows from $\wt{H}^{3}=\cH_t^{3}=1>0$; so it
  suffices to prove nefness.

  Let $\wt{S}\in|-K_{\wt{X}}|$ be a general anticanonical member. We first claim
  that $\wt{E}+2\wt{F}-2\wt{H}$ is not effective. Indeed, $|-K_{\wt{X}}|=|-g^{*}K_{X}|$
  is base-point free by \cite[Theorem 3.1]{KLPZ26}, so for general $\wt{S}$ the
  restriction $(\wt{E}+2\wt{F}-2\wt{H})|_{\wt{S}}$ would again be effective were
  $\wt{E}+2\wt{F}-2\wt{H}$ effective; but its intersection number with the nef
  divisor $\wt{H}|_{\wt{S}}$ of \Cref{prop:nefness-on-K3} equals $-4$, a
  contradiction. Since $2\wt{H}-\wt{S}\sim\wt{E}+2\wt{F}-2\wt{H}$, the exact sequence
  \[
    0\ \longrightarrow\ \cO_{\wt{X}}(\wt{E}+2\wt{F}-2\wt{H})
     \ \longrightarrow\ \cO_{\wt{X}}(2\wt{H})
     \ \longrightarrow\ \cO_{\wt{S}}\bigl(2\wt{H}|_{\wt{S}}\bigr)\ \longrightarrow\ 0
  \]
  yields an injection
  $H^{0}(\wt{X},\cO_{\wt{X}}(2\wt{H}))\hookrightarrow
   H^{0}(\wt{S},\cO_{\wt{S}}(2\wt{H}|_{\wt{S}}))$. By upper semicontinuity of cohomology,
  \[
    h^{0}\bigl(\wt{X},\cO_{\wt{X}}(2\wt{H})\bigr)\ \ge\ h^{0}(\sX_t,2\cH_t)
    \ =\ h^{0}\bigl(\bP^3,\cO_{\bP^3}(2)\bigr)\ =\ 10 ,
  \]
  while Kawamata--Viehweg vanishing and Riemann--Roch on the K3 surface $\wt{S}$ give
  \[
    h^{0}\bigl(\wt{S},\cO_{\wt{S}}(2\wt{H}|_{\wt{S}})\bigr)
    \ =\ \chi\bigl(\wt{S},\cO_{\wt{S}}(2\wt{H}|_{\wt{S}})\bigr)
    \ =\ \tfrac12\bigl(2\wt{H}|_{\wt{S}}\bigr)^{2}+2\ =\ 10 .
  \]
  Hence the restriction map is an isomorphism. As $2\wt{H}|_{\wt{S}}$ is globally
  generated by \cite{SD74,May72}, it follows that
  $\Bs|2\wt{H}|\cap\wt{S}=\emptyset$.

  Now let $\Gamma\subseteq\wt{X}$ be an irreducible curve. If $\Gamma$ is contracted
  by $g$, then $(\wt{H}\cdot\Gamma)\ge0$ because $\wt{\cH}$ is relatively nef over
  $\sX$. Otherwise $(-K_{\wt{X}}\cdot\Gamma)>0$, so $\Gamma$ meets the general member
  $\wt{S}$; by the previous paragraph $\Gamma\not\subseteq\Bs|2\wt{H}|$,
  and therefore $(\wt{H}\cdot\Gamma)\ge0$ as well. Thus $\wt{H}$ is nef.
\end{proof}

As a consequence, the line bundle $\wt{\cH}$ on $\wt{\sX}$ is relatively nef and big over $T$. By the basepoint-free theorem (cf. \cite[Theorem 3.3]{KM98}), $\wt{\cH}$ descends to an ample line bundle $\cG$ on its ample model over $T$,
\[\begin{tikzcd}[ampersand replacement=\&]
	{\wt{\sX}} \&\& {\sV} \\
	\& T
	\arrow["{\mtf{h}}", from=1-1, to=1-3]
	\arrow["{\wt{\pi}}"', from=1-1, to=2-2]
	\arrow["\phi", from=1-3, to=2-2]
\end{tikzcd}\]
where
\[
\sV\ \coloneqq \ \Proj_{\cO_T}\Big(\bigoplus_{m\geq 0} \wt{\pi}_* \wt{\cH}^{\otimes m}\Big).
\]
For any $0\neq t\in T$, the morphism $\wt{\sX}_t\rightarrow \sV_t$ contracts $\wt{\sE}_t$ to a $(2,2)$-complete intersection curve $\sC_t$ on $\sV_t\simeq \bP^3$ and contracts $\wt{\sF}_t$ to a smooth point $p_t\in \sC_t$, while $\cG_t\simeq \cO_{\bP^3}(1)$. We now consider the fibers over $0\in T$. To ease notation, we set
\begin{enumerate}
    \item $V\coloneqq \sV_0$;
    \item $h\colon\wt{X}\rightarrow V$ to be the restriction of $\mtf{h}$ to the central fiber $\wt{\sX}_0=\wt{X}$; and
    \item $G\coloneqq \cG_0$, a Cartier divisor on $V$.
\end{enumerate}

\begin{prop}\label{prop:contraction of E}
    Notation as above. Then the following hold:
    \begin{enumerate}
        \item[\textup{(1)}] the central fiber $V$ is isomorphic to $\bP^3$; 
        \item[\textup{(2)}]  $h$ contracts $\wt{E}$ to a curve $C$ on $V$ and contracts $\wt{F}$ to a point $p\in C$; and
        \item[\textup{(3)}]  $h$ is an isomorphism on $\wt{X}\setminus(\wt{E}\cup \wt{F})$.
    \end{enumerate}
\end{prop}

\begin{proof}
   Since $\wt{\sX}$ is klt and both $\wt{\cH}$ and $-K_{\wt{\sX}}$ are relatively nef and big over $T$, the Kawamata--Viehweg vanishing theorem gives
   \[
   R^i \wt{\pi}_* \wt{\cH}^{\otimes m}\ =\ 0\qquad\text{for all }i>0,\ m\in\bN.
   \]
   Hence, by cohomology and base change, $\wt{\pi}_* \wt{\cH}^{\otimes m}$ is locally free and satisfies
   \[
   \big(\wt{\pi}_* \wt{\cH}^{\otimes m}\big)\otimes k(0)\ \simeq\ H^0\big(\wt{X}, \wt{H}^{\otimes m}\big).
   \]
   It follows that
   \[
   V\ =\ \sV_0\ \simeq\ \Proj \Big(\bigoplus_{m\in \bN} H^0\big(\wt{X}, \wt{H}^{\otimes m}\big)\Big)
   \]
   is the ample model of $\wt{H}$ on $\wt{X}$; in particular, $V$ is normal projective and $h$ is birational. Since $h^*K_V=K_{\wt{X}}-\wt{E}-2\wt{F}\leq K_{\wt{X}}$ and $\wt{X}$ is klt, $V$ is also klt. As $-K_{\sV_t}\sim 4\cG_t$ for $0\neq t\in T$, we have $-K_V\sim 4G$; as $G$ is Cartier, $V\simeq \bP^3$ by \cite[Theorem 1.5]{ADL23}.

   Now consider the restriction $\mtf{h}|_{\wt{\sF}}$. Its image $\sP$ is an irreducible subscheme of $\sV$ whose fiber over each $0\neq t\in T$ is the point $p_t$; thus $\sP$ is a curve and $p\coloneqq \sP_0$ is a point. Similarly, the image $\sC$ of $\mtf{h}|_{\wt{\sE}}$ is a surface whose central fiber $C$ is a curve containing $p$. 
   
   Finally, we show that $h$ is an isomorphism on
   $\wt{X}\setminus(\wt{E}\cup\wt{F})$. First note that $\wt{H},\wt{E},\wt{F}$ form a
   basis of $N^1(\wt{X})_{\bQ}$: they are linearly independent, since already their restrictions to a general member $\wt{S}\in|-K_{\wt{X}}|$ are, and
   $\rho(\wt{X})=\rho(\wt{\sX}_t)=3$ by \cite[Lemma 2.9]{KLPZ26}.

   Let $\Gamma\subseteq\wt{X}$ be an irreducible curve contracted by
   $h$, so that $(\wt{H}\cdot\Gamma)=0$, and suppose that
   $\Gamma\not\subseteq\wt{E}\cup\wt{F}$. Then $\wt{E}$ and $\wt{F}$ are effective
   Cartier divisors not containing $\Gamma$, whence $(\wt{E}\cdot\Gamma)\ge0$ and
   $(\wt{F}\cdot\Gamma)\ge0$. Since $-K_{\wt{X}}=4\wt{H}-\wt{E}-2\wt{F}$ is nef, we
   obtain
   \[
     0\ \le\ (-K_{\wt{X}}\cdot\Gamma)\ =\ -(\wt{E}\cdot\Gamma)-2(\wt{F}\cdot\Gamma)\ \le\ 0 ,
   \]
   and therefore $(\wt{H}\cdot\Gamma)=(\wt{E}\cdot\Gamma)=(\wt{F}\cdot\Gamma)=0$. As
   these three classes span $N^1(\wt{X})_{\bQ}$, this forces $[\Gamma]=0$ in
   $N_1(\wt{X})_{\bQ}$, contradicting $(A\cdot\Gamma)>0$ for an ample divisor $A$ on
   $\wt{X}$. Hence every curve contracted by
   $h$ is contained in $\wt{E}\cup\wt{F}$, as desired.
\end{proof}

\begin{prop}\label{lem:not-plane-curve}
    The curve $C$ satisfies Serre's $S_1$ condition.
\end{prop}

\begin{proof}
    Suppose not. Then \Cref{lem:Hilbert-scheme} implies that $C$ is a non-reduced
    quartic curve with embedded points. Let $C'\subseteq C$ be its purification
    (cf.\ \cite[Tag 02OL]{Sta18}), obtained by removing the embedded points; since
    these are zero-dimensional, $C'$ is a (possibly non-reduced) plane quartic curve,
    of the same degree $4$, with $\Supp C'=\Supp C$. Let $\Pi\subseteq V\simeq\bP^3$
    be the plane containing $C'$, let $\wt{\Pi}\subseteq\wt{X}$ be its proper
    transform, and write $\nu\coloneqq h|_{\wt{\Pi}}\colon\wt{\Pi}\to\Pi$.
    Let $\wt{S}\in|-K_{\wt{X}}|$ be a general member and $\ove{S}\subseteq V$ its image.

    Over $T^{\circ}$ the curve $\sC_t$ is contained in $\ove{\sS}_t$; since $\sC$ is
    obtained by taking closures, the curve $C$ is contained in $\ove{S}$. As
    $\Pi\not\subseteq\ove{S}$, the restriction $\ove{S}|_{\Pi}$ is a divisor of degree
    $4$ on $\Pi\simeq\bP^2$ containing $C'$, which also has degree $4$; hence
    $\ove{S}|_{\Pi}=C'$ and $\Supp(\ove{S}\cap\Pi)=\Supp C'$.

    Consequently $\Supp\bigl(\wt{S}\cap\wt{\Pi}\bigr)$ is contained in
    $W\coloneqq\nu^{-1}\bigl(\Supp C'\bigr)$, which is a curve on $\wt{\Pi}$ not
    depending on the choice of $\wt{S}$. As $|-K_{\wt{X}}|$ is base-point free and
    $\wt{S}$ is general, $\wt{S}$ contains no component of $W$, so $\wt{S}\cap W$ is
    finite. On the other hand $\wt{S}|_{\wt{\Pi}}$ is an effective Cartier divisor
    on the surface $\wt{\Pi}$, hence of pure dimension one whenever it is non-empty.
    We conclude that $\wt{S}\cap\wt{\Pi}=\emptyset$, so
    $-K_{\wt{X}}|_{\wt{\Pi}}\equiv 0$ and the crepant morphism
    $g\colon\wt{X}\to X$ contracts the divisor $\wt{\Pi}$ to a point. This
    contradicts \cite[Theorem 1.3(1)]{Liu25}.
\end{proof}

\begin{prop}\label{prop:complete-intersection}
The curve $C\subseteq\bP^3$ is a $(2,2)$-complete intersection.
\end{prop}

\begin{proof}
Suppose not. By \Cref{lem:Hilbert-scheme} and \Cref{lem:not-plane-curve}, there are two
possibilities:
\begin{enumerate}
    \item either $C=C_1\cup_q C_2$ is the union of a plane cubic $C_1$ and a line $C_2$
    meeting at a point $q$, where $C_2$ is not contained in the plane $\Pi$ spanned
    by $C_1$,
    \item or $C$ is the ACM double plane curve arising as the limit of the former when
    $C_2$ approaches $\Pi$; in this case we write $\Pi$ for the supporting plane
    and $C_1$ for the degree-three part of $C$.
\end{enumerate}
Since $C\subseteq\ove{S}$ and $\ove{S}|_{\Pi}$ has degree $4$ on
$\Pi\simeq\bP^2$ while $C_1$ has degree $3$, the residual divisor $\ove{\Gamma}\ \coloneqq\ \ove{S}|_{\Pi}-C_1$ has degree $1$, hence is a line. Let $\wh{S}\to\ove{S}$ be the minimal resolution; it
factors as $\wh{S}\to\wt{S}\to\ove{S}$. Denote by $\Gamma$ the proper transform of
$\ove{\Gamma}$ on $\wh{S}$, and by $\wh{H},\wh{E},\wh{F}$ the pullbacks of
$\wt{H}|_{\wt{S}}$, $\wt{E}|_{\wt{S}}$, $\wt{F}|_{\wt{S}}$ respectively. In either case
$\Gamma^2=-2$ and $(\Gamma\cdot\wh{H})=1$, and we set
\[
a\coloneqq(\Gamma\cdot\wh{E}),\qquad b\coloneqq(\Gamma\cdot\wh{F}) .
\]
Since $\wh{F}$ is effective and $\Gamma$ is not a component of it, we have $b\geq0$. As
$\wh{H},\wh{E},\wh{F},\Gamma$ lie in the hyperbolic lattice $\NS(\wh{S})$, of signature
$(1,\rho-1)$, the determinant of their Gram matrix is non-positive, and is negative
precisely when the four classes are linearly independent.

Let $\wt{\Pi}\subseteq\wt{X}$ be the proper transform of $\Pi$ and
$\theta\colon\wt{\Pi}\to\Pi\simeq\bP^2$ the induced birational morphism, so that
$\wt{H}|_{\wt{\Pi}}=\theta^{*}\cO_{\bP^2}(1)$. Since $-K_{\wt{X}}$ is the pullback of
the very ample divisor $-K_X$ and $\wt{S}\in|-K_{\wt{X}}|$ is general, we have
$\wt{\Pi}\not\subseteq\wt{S}$, and by Bertini's theorem $\wt{S}\cap\wt{\Pi}$ is an
integral curve, whose image on $\Pi$ is $\ove{\Gamma}$. As
$\theta_{*}\big(\wt{S}|_{\wt{\Pi}}\big)=\ove{\Gamma}$ has degree
$1$, the difference $\theta^{*}\cO_{\bP^2}(1)-\wt{S}|_{\wt{\Pi}}$ is an effective
$\theta$-exceptional divisor, that is,
\[
\wt{S}|_{\wt{\Pi}}\ \leq\ \theta^{*}\cO_{\bP^2}(1) .
\]
Moreover, since $\wt{S}|_{\wt{\Pi}}$ is base-point free, either
\begin{enumerate}
    \item $\wt{S}|_{\wt{\Pi}}=\theta^{*}\cO_{\bP^2}(1)$, or
    \item $\wt{S}|_{\wt{\Pi}}$ is nef but not big, and defines a morphism
    $\wt{\Pi}\to\bP^1$.
\end{enumerate}
In the latter case, $-K_{\wt{X}}|_{\wt{\Pi}}=\wt{S}|_{\wt{\Pi}}$ is trivial on the
fibers of $\wt{\Pi}\to\bP^1$, so the crepant morphism $g\colon\wt{X}\to X$ contracts
the divisor $\wt{\Pi}$ onto a smooth rational curve $R$ with $(-K_X\cdot R)=1$. In
particular $X$ is singular along $R$, and therefore K-unstable by
\cite[Theorem 3.11]{KLPZ26}, a contradiction. In the former case, restricting
$-K_{\wt{X}}=4\wt{H}-\wt{E}-2\wt{F}$ to $\wt{\Pi}$ and pushing forward along $\theta$,
and using that $\wt{F}|_{\wt{\Pi}}$ is $\theta$-exceptional, we obtain
$\theta_{*}\big(\wt{E}|_{\wt{\Pi}}\big)=3\,\cO_{\bP^2}(1)$; hence by the projection
formula
\[
a\ =\ \big(\Gamma\cdot\wh{E}\big)
\ =\ \big(\theta^{*}\cO_{\bP^2}(1)\cdot\wt{E}|_{\wt{\Pi}}\big)
\ =\ \big(\cO_{\bP^2}(1)\cdot\theta_{*}(\wt{E}|_{\wt{\Pi}})\big)\ =\ 3 ,
\]
and the Gram matrix reads
\begin{equation*}
\left(
\begin{array}{c|cccc}
& \wh{H} & \wh{E} & \wh{F} & \Gamma \\ \hline \\[-1em]
\wh{H} & 4 & 4 & 0 & 1 \\
\wh{E} & 4 & 0 & 1 & 3 \\
\wh{F} & 0 & 1 & -2 & b \\
\Gamma & 1 & 3 & b & -2
\end{array}\right),
\qquad
\det\ =\ 16b^{2}+16b-31 .
\end{equation*}
This determinant is positive for every integer $b\geq1$, so $b=0$ and the determinant
equals $-31$. The four classes are therefore independent and span a lattice isometric to
$\Lambda_{\textup{I}}$ (cf.\ \Cref{eq:Lambda1}). Thus $[(S,-K_X|_S)]$ lies in the
Noether--Lefschetz divisor $\cD_{\Lambda_{\textup{I}}}$, and $X$ is K-unstable by
\Cref{prop:K-unstable-I}, a contradiction.
\end{proof}

\begin{rem}
    The idea behind \Cref{prop:complete-intersection} is the following. If $C_1$ meets
    $\ove{\Gamma}$ transversally at three distinct points, then, since
    $h\colon\wt{X}\to\bP^3$ is an isomorphism away from $\wt{E}\cup\wt{F}$, one reads off
    at once that either $a\geq3$ and $b\geq0$, or $a\geq2$ and $b\geq1$; the determinant
    of the Gram matrix then leaves only $(a,b)=(3,0)$ and $(a,b)=(2,1)$, which are
    excluded by \Cref{prop:K-unstable-I} and \Cref{prop:K-unstable-II} respectively. When
    $C_1$ and $\ove{\Gamma}$ meet with multiplicities this reasoning breaks down, which is
    why we argue instead through the linear system $\wt{S}|_{\wt{\Pi}}$; the latter
    argument has the additional advantage of forcing $b=0$ outright, so that the case
    $(a,b)=(2,1)$, and with it \Cref{prop:K-unstable-II}, never arises.
\end{rem}

\smallskip

\begin{proof}[Proof of \Cref{thm:K-ss-limit}]
    After shrinking $T$, we may assume that $\sV\to T$ is isomorphic to
    $\bP^3_T\to T$. Let $\sY\coloneqq \Bl_{\sC}\sV\to \sV$ be the blow-up along
    $\sC$; since $\sC$ is a $(2,2)$-complete intersection in $\bP^3_T$, the variety
    $\sY$ is a $(1,2)$-hypersurface in $(\bP^1\times\bP^3)_T$, and
    $\Bl_{\sC_0}\sV_0\simeq(\Bl_{\sC}\sV)_0$.

    For $t\neq 0$ the preimage of $\sP_t$ is a fiber of
    $(\bP^1\times\bP^3)_t\to\bP^3$. Taking its closure in $(\bP^1\times\bP^3)_T$
    yields a surface $\sR\subseteq\sY$ with $\sR_0$ again a fiber of
    $(\bP^1\times\bP^3)_0\to\bP^3$; after shrinking $T$, we may assume that
    $\sR\to T$ is a complete intersection of three divisors of class $\cO(0,1)$ in
    $(\bP^1\times\bP^3)_T$. Set
    \[
    \sU\ \coloneqq\ \Bl_{\sR}\sY\ \subseteq\ \Bl_{\sR}(\bP^1\times\bP^3)_T\ \simeq\
    (\Bl_p\bP^3\times\bP^1)_T,
    \]
    the proper transform of $\sY$ under
    $\Bl_{\sR}(\bP^1\times\bP^3)_T\to(\bP^1\times\bP^3)_T$; thus $\sU$ is a
    hypersurface in $(\Bl_p\bP^3\times\bP^1)_T$ and in particular Cohen--Macaulay,
    hence $S_2$. Moreover, since $\sU$ is regular in codimension $1$, it is normal. By
    adjunction, $-K_{\sU}$ is the restriction of the sum of the pullback of
    $\cO_{\bP^1}(1)$ and the ample divisor $2H-F$ from $\Bl_p\bP^3$, and is therefore ample over $T$. It is worth noting that $\sY_0$ is generically smooth along
    $\sR_0$ if and only if $\sU_0$ is irreducible, if and only if $\sU_0$ is the blowup of $\sY_0$ along $\sR_0$. 

    The families $\sX\to T$ and $\sU\to T$ are both normal, isomorphic in
    codimension one, and carry relatively ample anticanonical divisors; hence
    $\sX\simeq\sU$ over $T$. In particular $\sU_0=\sX_0=X$ is irreducible, so $\sY_0$ is generically smooth along $\sR_0$, which is equivalent
    to saying that $\sC_0$ has planar singularities. Therefore $X=\sU_0$ is the
    blowup of a fiber of the exceptional divisor of the blowup of $\bP^3$ along a
    $(2,2)$-complete intersection curve.
\end{proof}

\bigskip
\section{Variation of GIT quotients}\label{sec:GIT}

The aim of this section is to determine the GIT quotients of the space of pairs
$(Y,\ell_p)$, where $Y\subseteq\bP^1\times\bP^3$ is a $(1,2)$-hypersurface and
$\ell_p\subseteq Y$ is a fiber of the second projection, for all linearizations
$\xi+t\eta$. Everything here is an explicit Hilbert--Mumford computation.

\subsection{Setup}\label{sec:setup}

Let $\bP^3$ (resp.\ $\bP^1$) be a $3$-dimensional (resp.\ $1$-dimensional) projective
space with homogeneous coordinates $[x,y,z,w]$ (resp.\ $[u_0,u_1]$). For a point
$p\in\bP^3$ we denote by $\ell_p$ the fiber of the second projection
$\bP^1\times\bP^3\rightarrow\bP^3$ over $p$, and call it a \emph{line}. Let
$\cL\subseteq(\bP^1\times\bP^3)\times\bP^3$ be the universal family of lines, and let
\[
p_1\colon(\bP^1\times\bP^3)\times\bP^3\rightarrow\bP^1\times\bP^3,
\qquad
p_2\colon(\bP^1\times\bP^3)\times\bP^3\rightarrow\bP^3
\]
be the two projections. Twisting the exact sequence
\[
0\longrightarrow\cI_{\cL}\longrightarrow\cO_{(\bP^1\times\bP^3)\times\bP^3}
\longrightarrow\cO_{\cL}\longrightarrow 0
\]
by $p_1^*\cO_{\bP^1\times\bP^3}(1,2)$ and pushing forward along $p_2$ yields an exact
sequence of locally free sheaves on $\bP^3$:
\[
0\longrightarrow
p_{2*}\big(\cI_{\cL}\otimes p_1^*\cO_{\bP^1\times\bP^3}(1,2)\big)
\longrightarrow
p_{2*}\big(p_1^*\cO_{\bP^1\times\bP^3}(1,2)\big)
\longrightarrow
p_{2*}\big(\cO_{\cL}\otimes p_1^*\cO_{\bP^1\times\bP^3}(1,2)\big)
\longrightarrow 0.
\]
Let $\cE$ be the dual of $p_{2*}\big(\cI_{\cL}\otimes p_1^*\cO_{\bP^1\times\bP^3}(1,2)\big)$,
a rank-$18$ vector bundle on $\bP^3$, and let $\pi\colon\bP\cE\rightarrow\bP^3$ be the
natural projection. Since
\[
p_{2*}\big(p_1^*\cO_{\bP^1\times\bP^3}(1,2)\big)\simeq V_{1,2}\otimes\cO_{\bP^3},
\qquad
V_{1,2}:=H^0\big(\bP^1\times\bP^3,\cO_{\bP^1\times\bP^3}(1,2)\big),
\]
the variety $\bP\cE$ is a projective subbundle of
\[
\bP\big(V_{1,2}^{*}\otimes\cO_{\bP^3}\big)
\simeq\bP^3\times\bP V_{1,2}^{*}
\simeq\bP^3\times\bP^{19}
\]
over $\bP^3$.

A closed point of $\bP\cE$ corresponds to a pair $(Y,\ell_p)$, where $Y\subseteq
\bP^1\times\bP^3$ is a hypersurface of class $\cO_{\bP^1\times\bP^3}(1,2)$ and
$\ell_p\subseteq Y$ is a line; when $Y$ is smooth it is a Fano threefold in family
\textnumero 2.25. There is a universal family
$g\colon(\mts{Y},\cL_{\bP\cE})\rightarrow\bP\cE$ whose fiber over $[(Y,\ell_p)]$ is
$(Y,\ell_p)$. Let $h\colon\mts{X}:=\Bl_{\cL_{\bP\cE}}\mts{Y}\rightarrow\mts{Y}$ be the
blow-up. Then $f:=g\circ h\colon\mts{X}\rightarrow\bP\cE$ is a flat family whose general
fiber is a smooth Fano threefold in family \textnumero 3.11.

Write $\eta:=\pi^*\cO_{\bP^3}(1)$ and $\xi:=\cO_{\bP\cE}(1)$, and let $h$ denote the
class of a hyperplane in $\bP^3$. Then $\Pic(\bP\cE)\simeq\bZ\cdot\eta\oplus\bZ\cdot\xi$.
Moreover $\xi$ is the restriction of $\cO_{\bP^3\times\bP^{19}}(0,1)$ to $\bP\cE$, and
$\bP\cE\rightarrow\bP^{19}$ has relative dimension $1$; hence the two extremal rays of
the nef cone of $\bP\cE$ are generated by $\xi$ and $\eta$. We define the \emph{slope} of
a line bundle $a\xi+b\eta$ with $a\neq0$ to be $t=b/a$.

Consider the natural action of $G:=\PGL(2)\times\PGL(4)$ on $\bP^3\times\bP^{19}$, which
leaves $\bP\cE$ invariant. For a rational number $t>0$ we call a pair $(Y,\ell_p)$
\emph{$t$-(semi/poly)stable} if $[(Y,\ell_p)]\in\bP\cE$ is GIT (semi/poly)stable with
respect to the polarization $\cL_t:=\xi+t\eta$, and we denote by $\bP\cE^{\sst}(t)$ the
open subset of $t$-semistable points.

\begin{defn}\label{gitdefn}
    For any $t>0$, define the quotient stack
    \[
    \MM^{\GIT}(t):=\big[\bP\cE^{\sst}(t)/G\big],
    \]
    which is an Artin stack admitting the GIT quotient
    \[
    \ove{\fM}^{\GIT}(t):=\bP\cE\sslash_t G
    \]
    as a good moduli space.
\end{defn}

The set $\bP\cE^{\sst}(t)$, and hence $\ove{\fM}^{\GIT}(t)$, is constant as $t$ varies in a
chamber; for an open interval $I$ of slopes contained in a single chamber we write
$\ove{\fM}^{\GIT}(I)$ for this common quotient.

Let $(Y,\ell_p)\in\bP\cE$ be a pair, with $Y$ defined by
$u_0Q_0(x,y,z,w)+u_1Q_1(x,y,z,w)=0$ for quadratic forms $Q_i$ on $\bP^3$. Set
$C_Y:=\bV(Q_0,Q_1)\subseteq\bP^3$. Since $(Y,\ell_p)\in\bP\cE$, the point $p$ lies on
$C_Y$. Note that $C_Y$ need not be a curve: if $Q_0$ and $Q_1$ have a common factor, then
$C_Y$ is not a complete intersection.

\subsection{Stability conditions and VGIT quotient}

Throughout this subsection we normalize $p=[0,0,0,1]$, so that the defining equation of
$Y$ takes the form
\begin{equation}\label{eq:defining}
u_0\big(w\ell_0(x,y,z)+q_0(x,y,z)\big)
+u_1\big(w\ell_1(x,y,z)+q_1(x,y,z)\big)=0,
\end{equation}
with $\ell_i$ linear and $q_i$ quadratic. If $p$ is a smooth point of $C_Y$, we may
further assume $\ell_0=x$ and $\ell_1=y$, so that $Y$ is defined by
\begin{equation}\label{eq:smooth}
u_0\big(wx+q_0(x,y,z)\big)+u_1\big(wy+q_1(x,y,z)\big)=0.
\end{equation}
If moreover the component of $C_Y$ containing $p$ is not a line, then after a change of
coordinates on $\bP^1$ we may assume that $q_0$ contains a $z^2$-term while $q_1$ does
not; that is, $Y$ is defined by
\begin{equation}\label{eq:simple}
u_0\big(wx+z^2+q'_0(x,y,z)\big)+u_1\big(wy+q'_1(x,y,z)\big)=0,
\end{equation}
where neither $q'_0$ nor $q'_1$ contains a $z^2$-term.

\begin{remark}\label{rem:inflection point}
    \textup{For any elliptic normal curve $C\subseteq\bP^3$ and any point $p\in C$ there
    is a plane $\Lambda_p$ with local intersection multiplicity
    $(\Lambda_p\cdot C)_p\geq3$; we call $p$ an \emph{inflection point} of $C$ if
    $(\Lambda_p\cdot C)_p=4$. If $Y$ is defined by \eqref{eq:simple}, then the plane
    $\Lambda_p$ with $(\Lambda_p\cdot C_Y)_p\geq3$ is precisely $\bV(y)$, and $p$ is an
    inflection point of $C_Y$ if and only if $q'_1$ has no $xz$-term.}
\end{remark}

\begin{defn}[Diagonal 1-PS]
    A one-parameter subgroup (1-PS) $\lambda\colon\bG_m\rightarrow\SL(2)\times\SL(4)$ is
    \emph{diagonal of weight $(\alpha,-\alpha;\beta_0,\beta_1,\beta_2,\beta_3)$} if, as
    a pair of matrices,
    \[
    \lambda(t)=\big(\diag(t^{\alpha},t^{-\alpha}),
    \diag(t^{\beta_0},t^{\beta_1},t^{\beta_2},t^{\beta_3})\big),
    \]
    where $\alpha,\beta_i\in\bZ$ and $\beta_0+\beta_1+\beta_2+\beta_3=0$.
\end{defn}

\begin{lemma}[Hilbert--Mumford weight]\label{lem:HM weight}
    Let $t>0$ be a rational number and let $(Y,\ell_p)\in\bP\cE$ be a pair. Then for any
    diagonal 1-PS $\lambda$ of weight $(\alpha,-\alpha;\beta_0,\beta_1,\beta_2,\beta_3)$,
    \[
    \mu^{t}(Y,\ell_p;\lambda)\ =\ \mu(Y;\lambda)+t\,\mu(\ell_p;\lambda),
    \qquad
    \mu(\ell_p;\lambda)=\beta_3,
    \]
    where, writing $(x_0,x_1,x_2,x_3)=(x,y,z,w)$,
    \[
    \mu(Y;\lambda)\ =\ \max\big\{
    -\alpha-\beta_i-\beta_j,\ \alpha-\beta_k-\beta_m
    \big\},
    \]
    the maximum being taken over those monomials $u_0x_ix_j$ and $u_1x_kx_m$ whose
    coefficients in \eqref{eq:defining} are nonzero.
\end{lemma}

In what follows we abbreviate $\mu(Y):=\mu(Y;\lambda)$ and
$\mu(\ell_p):=\mu(\ell_p;\lambda)$ when $\lambda$ is clear from the context.

\begin{remark}[Extremal ray of the $G$-ample cone, I]\label{rem:extremal}
    \textup{For $t=0$ one can still define (semi/poly)stability of points of $\bP\cE$ by
    extending the Hilbert--Mumford criterion. A pair $(Y,\ell_p)$ is $0$-semistable if
    and only if $[Y]\in\bP^{19}$ is GIT-semistable under the $G$-action, and $[Y]$ is
    semistable if and only if $Y$ is irreducible with at worst $A_1$-singularities, or
    equivalently $C_Y$ is a complete intersection with at worst nodal singularities. The
    GIT quotient
    $\ove{\fM}^{\GIT}(0):=\bP^{19}\sslash G\simeq\Gr(2,10)\sslash\PGL(4)$
    is then isomorphic to $\bP^1$.}
\end{remark}

\begin{lemma}\label{lem: necessary conditions}
    Let $t>0$ be a rational number and let $(Y,\ell_p)$ be a $t$-semistable pair. Then:
    \begin{enumerate}
        \item[\textup{(1)}] $C_Y$ is a complete intersection curve in $\bP^3$;
        \item[\textup{(2)}] $p$ is a smooth point of $C_Y$; and
        \item[\textup{(3)}] the component of $C_Y$ containing $p$ has degree at least $3$.
    \end{enumerate}
\end{lemma}

\begin{proof}
    We keep the notation and normalizations introduced above.

    (1) If $C_Y$ is not a complete intersection, then $\{Q_0=0\}$ and $\{Q_1=0\}$ share a
    component. If both quadrics are reducible we may assume $Q_0=xz$ and $Q_1=xy$; taking
    $\lambda$ of weight $(0,0;-1,1,1,-1)$ gives
    \[
    \mu^{t}(Y,\ell_p;\lambda)=\mu(Y)+t\mu(\ell_p)\leq-t<0
    \]
    for every $t>0$. If instead the quadrics are irreducible we may assume
    $Q_0=Q_1=xy+zw$ or $Q_0=Q_1=x^2+zw$; taking $\lambda$ of weight $(0,0;0,0,1,-1)$
    gives $\mu^{t}(Y,\ell_p;\lambda)\leq-t<0$.

    (2) If $p$ is a singular point of $C_Y$, then $\ell_0$ and $\ell_1$ are proportional,
    so after a change of coordinates on $\bP^1$ we may assume $\ell_0=0$. Taking
    $\lambda$ of weight $(-2,2;1,1,1,-3)$ gives
    \[
    \mu^{t}(Y,\ell_p;\lambda)=\mu(Y)+t\mu(\ell_p)\leq-3t<0
    \]
    for every $t>0$.

    (3) It suffices to rule out the cases where the component $C$ of $C_Y$ containing $p$
    is a line or a conic. If $C$ is a line, then $C=\bV(x,y)$, so neither $q_0$ nor $q_1$
    contains a $z^2$-term; taking $\lambda$ of weight $(0,0;1,1,-1,-1)$ gives
    \[
    \mu^{t}(Y,\ell_p;\lambda)=\mu(Y)+t\mu(\ell_p)\leq-t<0
    \]
    for every $t>0$. If $C$ is a conic, then $C_Y$ is contained in a non-normal quadric;
    after a change of coordinates on $\bP^1$ we may assume that $wx+q_0(x,y,z)$ defines
    such a quadric, so that $x\mid q_0$. Taking $\lambda$ of weight $(-2,2;3,-1,-1,-1)$
    gives $\mu^{t}(Y,\ell_p;\lambda)\leq-t<0$ for every $t>0$.
\end{proof}

We denote by $\bP\cE^{\circ}\subseteq\bP\cE$ the subset of pairs satisfying the three
conditions of \Cref{lem: necessary conditions}; it is an open $G$-invariant subset.
Note that if $(Y,\ell_p)\in\bP\cE^{\circ}$ is written in the form \eqref{eq:simple}, then
\Cref{lem: necessary conditions}(3) forces $q'_1$ to have a nonzero $xz$-term or a
nonzero $x^2$-term.

\begin{lemma}\label{lem:unstable above 1/3}
    Let $(Y,\ell_p)$ be a pair parametrized by $\bP\cE^{\circ}$. Then $(Y,\ell_p)$ is
    $t$-unstable for every rational $t>\tfrac13$.
\end{lemma}

\begin{proof}
    Taking $\lambda$ of weight $(1,-1;1,3,-1,-3)$ gives
    \[
    \mu^{t}(Y,\ell_p;\lambda)=\mu(Y)+t\mu(\ell_p)\leq 1-3t<0
    \]
    for every $t>\tfrac13$.
\end{proof}

\begin{lem}\label{lem:A3-unstable}
    Let $(Y,\ell_p)\in\bP\cE^\circ$ be a pair. If $C_Y$ has an $A_3$-singularity,
    then $(Y,\ell_p)$ is $t$-unstable for every $t<1/3$.
\end{lem}

\begin{proof}
    The pencil defining $C_Y$ contains a smooth quadric. Indeed, if it
    contains a union of two planes, then every component of $C_Y$ has
    degree at most two, contradicting \Cref{lem: necessary conditions}(3). If it consists entirely of quadric cones, then the cones have a common vertex and their base locus is a union of four lines, giving the same
    contradiction.

    Let $Q\simeq\bP^1\times\bP^1$ be a smooth quadric containing $C_Y$. Since an $A_3$-singularity has
    $\delta$-invariant two, $C_Y$ cannot be irreducible. By \Cref{lem: necessary conditions}(3) again, after exchanging the two rulings we have $C_Y=C_3\cup L$, where $C_3$ and $L$ have bidegrees
    $(2,1)$ and $(0,1)$, respectively, and $p\in C_3$. We may therefore choose coordinates such that the pencil is generated
    by
    $Q_0=xw-yz$ and
    $Q_1=xy+z^2+ax^2+bxz$.
    Take the $1$-PS $\lambda$ of weight
    $(1,-1;3,-1,1,-3)$. Then $\mu(Y;\lambda)\leq-1$, while
    $\mu(\ell_p;\lambda)\leq3$ for every $p\in\bP^3$. Hence
    \[
        \mu_t(Y,\ell_p;\lambda)\leq-1+3t<0
    \]
    whenever $t<1/3$, proving the assertion.
\end{proof}

\begin{prop}[Extremal ray of the $G$-ample cone, II]\label{prop:extremal ray II}
    Let $(Y,\ell_p)$ be a pair parametrized by $\bP\cE^{\circ}$ and assume that $p$ is
    not an inflection point of $C_Y$. Then:
    \begin{enumerate}
        \item[\textup{(1)}] $(Y,\ell_p)$ is $\tfrac13$-semistable; and
        \item[\textup{(2)}] up to the $G$-action, the unique $\tfrac13$-polystable pair in the
        $S$-equivalence class of $(Y,\ell_p)$ is
        \[
        \big(\bV\big(u_0(wx+z^2)+u_1(wy+xz)\big),\ \ell_{[0,0,0,1]}\big),
        \]
    \end{enumerate} which corresponds to a twisted cubic curve union a tangent line together with an arbitrary smooth point on the cubic. In particular, $\ove{\fM}^{\GIT}(\tfrac13)$ is a single point.
\end{prop}

\begin{proof}
    Let $\lambda_0$ be the diagonal 1-PS of weight $(1,-1;1,3,-1,-3)$. Then
    $\mu^{1/3}(Y,\ell_p;\lambda_0)=0$, and the limit of $(Y,\ell_p)$ under $\lambda_0$ is
    the pair $(Y_0,\ell_p)$ with $Y_0$ defined by
    \[
    u_0(wx+z^2)+u_1(wy+a\,xz)=0
    \]
    for some $a\neq0$; after rescaling we may assume $a=1$.

    We claim that $(Y_0,\ell_p)$ is $\tfrac13$-polystable. If it were
    $\tfrac13$-unstable, then by \cite[Theorem 3.4]{Kem78} there would be a destabilizing
    1-PS $\lambda$ commuting with $\lambda_0$; such a $\lambda$ is diagonal, say of weight
    $(\alpha_0,\alpha_1;\beta_0,\beta_1,\beta_2,\beta_3)$. Then
    \[
    \begin{aligned}
    \mu^{1/3}(Y_0,\ell_p;\lambda)
    &=\max\big\{
    -\alpha_0-\beta_0-\beta_3,\
    -\alpha_0-2\beta_2,\
    -\alpha_1-\beta_1-\beta_3,\
    -\alpha_1-\beta_0-\beta_2
    \big\}+\tfrac13\beta_3\\
    &\geq-\tfrac1{12}\big(
    4(\alpha_0+\beta_0+\beta_3)
    +2(\alpha_0+2\beta_2)
    +5(\alpha_1+\beta_1+\beta_3)
    +(\alpha_1+\beta_0+\beta_2)\big)+\tfrac13\beta_3\\
    &=-\tfrac13\beta_3+\tfrac13\beta_3=0,
    \end{aligned}
    \]
    using $\alpha_0+\alpha_1=\beta_0+\beta_1+\beta_2+\beta_3=0$. Equality holds precisely
    when the weight is proportional to $(1,-1;1,3,-1,-3)$, which proves the claim and
    hence (2).

    Assertion (1) follows: since $(Y,\ell_p)$ degenerates to $(Y_0,\ell_p)$ under
    $\lambda_0$ with $\mu^{1/3}(Y,\ell_p;\lambda_0)=0$, and $(Y_0,\ell_p)$ is
    $\tfrac13$-semistable, so is $(Y,\ell_p)$.
\end{proof}

\begin{lemma}\label{lem:cusp unstable}
    Let $(Y,\ell_p)$ be a pair parametrized by $\bP\cE^{\circ}$. If $C_Y$ has a simple
    cusp, then $(Y,\ell_p)$ is $t$-unstable for every $t<\tfrac27$.
\end{lemma}

\begin{proof}
    Up to projective isomorphism, there is a unique cuspidal quartic complete intersection in $\bP^3$, so we may assume that
    $C_Y=V(wx+z^2,wy+x^2)$, whose cusp is $q=[0,1,0,0]$. After changing
    coordinates on $\bP^1$, the hypersurface $Y$ is therefore defined by
    $u_0(wx+z^2)+u_1(wy+x^2)=0$. Take the $1$-PS $\lambda$ of weight
    $(-4,4;-1,-9,3,7)$. Then $\mu(Y;\lambda)=-2$. For any
    $p\in\bP^3$, its contribution is bounded above by the largest weight
    on $\bP^3$, so
    $\mu(\ell_p;\lambda)\leq\max\{-1,-9,3,7\}=7$. Consequently,
    \[
        \mu_t(Y,\ell_p;\lambda)
        \ =\ \mu(Y;\lambda)+t\mu(\ell_p;\lambda)
        \ \leq\  -2+7t\ <\ 0
    \]
    whenever $t<2/7$. Hence $(Y,\ell_p)$ is $t$-unstable.
\end{proof}

\begin{lemma}\label{lem:inflection unstable}
    Let $(Y,\ell_p)$ be a pair parametrized by $\bP\cE^{\circ}$. If $p$ is an inflection
    point of $C_Y$, then $(Y,\ell_p)$ is $t$-unstable for every $t>\tfrac27$.
\end{lemma}

\begin{proof}
    By \Cref{rem:inflection point}, the form $q'_1$ in \eqref{eq:simple} has no
    $xz$-term. Taking $\lambda$ of weight $(4,-4;1,9,-3,-7)$ gives
    \[
    \mu^{t}(Y,\ell_p;\lambda)=\mu(Y)+t\mu(\ell_p)\leq 2-7t<0
    \]
    for every $t>\tfrac27$.
\end{proof}

\begin{lemma}\label{lem:cusp inflection polystable}
    Let $(Y_0,\ell_p)\in\bP\cE$ be a pair such that $C_{Y_0}$ is a cuspidal curve and $p$
    is an inflection point of $C_{Y_0}$. Then $(Y_0,\ell_p)$ is $t$-semistable if and
    only if $t=\tfrac27$; moreover, $(Y_0,\ell_p)$ is $\tfrac27$-polystable.
\end{lemma}

\begin{proof}
    The argument is the same as in \Cref{prop:extremal ray II}(2), with
    $\lambda_0$ of weight $(-4,4;-1,-9,3,7)$.
\end{proof}

Summarizing, we obtain the following description of the VGIT quotients.
\begin{theorem}[Wall crossing and VGIT quotients]\label{thm:wall crossing}
    Let $(Y,\ell_p)\in\bP\cE^{\circ}$. Then the following hold.
    \begin{enumerate}
        \item[\textup{(1)}] If $0<t<\tfrac27$, then $(Y,\ell_p)$ is $t$-semistable if
        and only if $C_Y$ is GIT-semistable, or equivalently, $Y$ is
        GIT-semistable.
        \item[\textup{(2)}] If $\tfrac27<t<\tfrac13$, then $(Y,\ell_p)$ is
        $t$-semistable if and only if $C_Y$ has at worst
        $A_2$-singularities and $p$ is not an inflection point of $C_Y$.
    \end{enumerate}
    In particular, there is a wall-crossing diagram
    \begin{equation}\nonumber
    \xymatrix @R=.07in @C=.07in{
      &   & &
      \ove{\fM}^{\GIT}(0,\tfrac{2}{7})
      \ar[rdd]_{\iota}\ar[ldd]_{\phi}\ar@{-->}[rr]
      & &
      \ove{\fM}^{\GIT}(\tfrac{2}{7},\tfrac{1}{3})
      \ar[ldd]^{\psi}\ar[rdd]^{\rho}
      & \\
      &&&&&& \\
      & &
      \bP^1\simeq\ove{\fM}^{\GIT}(0)
      & &
      \ove{\fM}^{\GIT}(\tfrac{2}{7})
      & &
      \ove{\fM}^{\GIT}(\tfrac{1}{3})=\Spec\bC,
    }
    \end{equation}
    where
    \begin{itemize}
        \item $\phi$ is generically a $\bP^1$-fibration;
        \item $\iota$ contracts the curve parametrizing pairs
        $(Y,\ell_p)$ for which $p$ is an inflection point of $C_Y$; and
        \item $\psi$ is an isomorphism.
    \end{itemize}
\end{theorem}

\begin{proof}
    By \Cref{rem:extremal,prop:extremal ray II}, the $G$-ample cone is
    the interval $[0,\tfrac13]$. By
    \Cref{lem: necessary conditions,lem:A3-unstable} and the closedness
    of the unstable locus, a pair is $t$-unstable for every
    $0<t<\tfrac13$ if $C_Y$ has singularities worse than $A_2$, if $p$
    is singular on $C_Y$, or if the component of $C_Y$ containing $p$
    has degree at most two.

    Suppose now that $C_Y$ has at worst nodal singularities and that
    $p$ is not an inflection point. The pair is semistable at $t=0$ by
    \Cref{rem:extremal} and at $t=\tfrac13$ by
    \Cref{prop:extremal ray II}. Since the Hilbert--Mumford function is
    linear in $t$, it is $t$-semistable for every
    $0\leq t\leq\tfrac13$.

    It remains to consider pairs for which either $p$ is an inflection
    point or $C_Y$ has an $A_2$-singularity. The preceding lemmas show
    that both loci contribute precisely to the wall $t=\tfrac27$:
    inflection-point pairs are unstable for $t>\tfrac27$, while
    cuspidal pairs are unstable for $t<\tfrac27$, and their limits at
    the wall are polystable. Together with the endpoint criteria and
    the linearity of the Hilbert--Mumford function, this proves the two
    stated chamber descriptions.

    The wall-crossing morphisms are those induced by VGIT. The morphism
    $\phi$ is generically a $\bP^1$-fibration, while $\iota$ contracts
    the locus of inflection-point pairs to its polystable limit at the
    wall. On the other side, every newly semistable cuspidal pair has
    the same wall limit, so $\psi$ is an isomorphism.
\end{proof}

\bigskip
\section{K-moduli space and classification}\label{sec:Kmod}

In this section, we prove that the K-moduli space of family \textnumero3.11 is the GIT
quotient at slope $t_0=\tfrac{5}{21}$, and read off from this which members are
K-semistable. The two remaining ingredients are the computation of the CM line bundle and the unobstructedness of deformations. The geometric input is \Cref{thm:K-ss-limit}, which asserts that every K-semistable member arises as the blowup $\Bl_{\ell_p}Y$ for a pair $(Y,\ell_p)$, so that the universal family over $\bP\cE$ already contains all of them.

\subsection{Computation of the CM line bundle}

Recall from \Cref{sec:setup} the families and morphisms
$f\colon\mts{X}\xrightarrow{\ g\ }\mts{Y}\xrightarrow{\ h\ }\bP\cE$.
In this subsection we compute the slope $t=b/a$ of the CM line bundle
\[
\lambda_{\CM,f}\ =\ -f_*\big(-K_{\mts{X}/\bP\cE}\big)^4\ =\ a\xi+b\eta,
\]
where $\xi=\cO_{\bP\cE}(1)$ and $\eta=\pi^*\cO_{\bP^3}(1)$. By the construction of
$\mts{Y}$ one has
\[
K_{\mts{Y}/\bP\cE}\ =\ \big(K_{\bP^1\times\bP^3\times\bP\cE}+\mts{Y}\big)\big|_{\mts{Y}}
\ =\ \cO_{\bP^1\times\bP^3\times\bP\cE}(-1,-2,1)\big|_{\mts{Y}},
\]
so that, writing $E=\bP\big(N^{*}_{\cL/\mts{Y}}\big)$ for the $g$-exceptional divisor,
\begin{equation}\label{eq:relative canonical}
K_{\mts{X}/\bP\cE}\ =\ f^{*}\big(\cO_{\bP^1\times\bP^3\times\bP\cE}(-1,-2,1)\big|_{\mts{Y}}\big)+E .
\end{equation}
To determine $a$ and $b$ we test $f$ against two subfamilies.

\subsubsection{Testing family 1}

Consider the family $f_1$ obtained from the pull-back diagram
\[
\xymatrix{
\mts{X}_1\ar@{^{(}->}[r]\ar[d]_{f_1} & \mts{X}\ar[d]^{f}\\
C_1\ar@{^{(}->}[r] & \bP\cE
}
\]
where $C_1\subseteq\bP\cE$ is a general rational curve of the form
$\{p\}\times C\subseteq\bP^3\times\bP^{19}$ with $C$ a line, that is, $\bV(C)$ cut out by
linear forms. Then $C_1$ lies in a fiber of $\bP\cE\to\bP^3$, so
\[
(C_1\cdot\xi)=1,\qquad (C_1\cdot\eta)=0 .
\]
The curve $C_1$ parametrizes a general pencil $\{Y_t\}_{t\in C_1}$ of hypersurfaces of
bidegree $(1,2)$ in $\bP^1\times\bP^3$ such that $p\in\Bs(C_{Y_t})$ for every $t\in C_1$.

\begin{lemma}\label{lem:testing family 1}
    Let $\lambda_{\CM,f_1}$ be the CM line bundle of the family $f_1$. Then
    \[ \deg\lambda_{\CM,f_1}
\ =\ (C_1\cdot\lambda_{\CM,f})
\ =\ -\bigl(K_{\sX_1/C_1}^4\bigr)
\ =\ 84 .
    \]
\end{lemma}

\begin{proof}
    Recall that $h_1\colon\mts{X}_1\to\mts{Y}_1$ is the blow-up along ${\cL}_1$. The
    projection ${\cL}_1\to C_1$ is a trivial $\bP^1$-bundle, so
    ${\cL}_1\simeq\bP^1\times C_1$ and
    \[
    \cO_{\bP^1\times\bP^3\times C_1}(n_1,n_2,n_3)\big|_{{\cL}_1}
    \ \simeq\ \cO_{{\cL}_1}(n_1,n_3).
    \]
    By \eqref{eq:relative canonical},
    \[
    K_{\mts{X}_1/C_1}\ =\ h_1^{*}\cO_{\mts{Y}_1}(-1,-2,1)+E_1 .
    \]

    Since ${\cL}_1\simeq\bP^1\times\{p\}\times C_1$, it is a complete intersection in
    $\bP^1\times\bP^3\times C_1$ of three divisors of class
    $\cO_{\bP^1\times\bP^3\times C_1}(0,1,0)$, whence
    \[
    N_{{\cL}_1/\bP^1\times\bP^3\times C_1}
    \ \simeq\ \cO_{\bP^1\times\bP^3\times C_1}(0,1,0)^{\oplus3}
    \ \simeq\ \cO_{{\cL}_1}^{\oplus3}.
    \]
    Combining the exact sequence
    \[
    0\longrightarrow N_1:=N_{{\cL}_1/\mts{Y}_1}
    \longrightarrow N_{{\cL}_1/\bP^1\times\bP^3\times C_1}
    \longrightarrow N_{\mts{Y}_1/\bP^1\times\bP^3\times C_1}\big|_{{\cL}_1}
    \longrightarrow 0
    \]
    with the isomorphism
    $N_{\mts{Y}_1/\bP^1\times\bP^3\times C_1}\simeq\cO_{\mts{Y}_1}(1,2,1)$ gives
    \[
    c_1(N_1)=\cO_{{\cL}_1}(-1,-1),
    \qquad
    c_2(N_1)=-\big(\cO_{{\cL}_1}(1,1)\cdot\cO_{{\cL}_1}(-1,-1)\big)=2 .
    \]
    Writing $\zeta_1$ for the class $\cO_{\bP N_1^{*}}(1)$ on $E_1\simeq\bP N_1^{*}$, we
    obtain
    \[
    \zeta_1^{2}=-c_1(N_1)\zeta_1-c_2(N_1),
    \qquad
    \zeta_1^{3}=\big(c_1(N_1)^{2}-c_2(N_1)\big)\zeta_1=0 ,
    \]
    the last equality because $c_1(N_1)^{2}=\cO_{{\cL}_1}(1,1)^{2}=2=c_2(N_1)$.

    Finally, $h_1^{*}\cO_{\mts{Y}_1}(0,0,1)$ is pulled back from $C_1$, so its square
    vanishes and only one cross term survives:
    \[
    -\big(K_{\mts{X}_1/C_1}^{4}\big)
    \ =\ -\big(h_1^{*}\cO_{\mts{Y}_1}(-1,-2,0)+E_1\big)^{4}
    +4\big(\!-\!K_{\mts{X}_1/C_1}\big)^{3}\cdot h_1^{*}\cO_{\mts{Y}_1}(0,0,1),
    \]
    and the second term equals $4\vol(-K_{X_t})=4\cdot28$. Expanding the first term and
    using $\big(h_1^{*}D^{2}\cdot E_1^{2}\big)=-\big(D^{2}\cdot{\cL}_1\big)$ and $(h_1^*D\cdot E_1^3)=-\bigl(D|_{\cL_1}\cdot c_1(N_1)\bigr)$ for
    $D=\cO_{\mts{Y}_1}(1,2,0)$, we get
    \[
    \begin{split}
    -\big(K_{\mts{X}_1/C_1}^{4}\big)
    &=-\big(\cO_{\mts{Y}_1}(1,2,0)^{4}\big)
      -6\big(h_1^{*}\cO_{\mts{Y}_1}(1,2,0)^{2}\cdot E_1^{2}\big)
      +4\big(h_1^{*}\cO_{\mts{Y}_1}(1,2,0)\cdot E_1^{3}\big)
      -\big(E_1^{4}\big)+4\cdot28\\
    &=-4\cdot2^{3}
      +6\big(\cO_{\mts{Y}_1}(1,2,0)^{2}\cdot{\cL}_1\big)
      -4\big(\cO_{\mts{Y}_1}(1,2,0)|_{{\cL}_1}\cdot c_1(N_1)\big)
      +\big(\zeta_1^{3}\big)+4\cdot28\\
    &=-32+0-4\cdot(-1)+0+112\ =\ 84 . \qedhere
    \end{split}
    \]
\end{proof}

\subsubsection{Testing family 2}

Consider now the family $f_2$ obtained from the diagram
\[
\xymatrix{
\mts{X}_2\ar@{^{(}->}[r]\ar[d]_{f_2} & \mts{X}\ar[d]^{f}\\
C_2\ar@{^{(}->}[r] & \bP\cE
}
\]
where $C_2=\Gamma\times[Y]\subseteq\bP\cE$ for a fixed $[Y]\in\bP^{19}$, the class of a
divisor $Y\subseteq\bP^1\times\bP^3$ of bidegree $(1,2)$, and $\Gamma\subseteq\bP^3$ is
the associated base curve: writing $Y=\{u_0Q_0+u_1Q_1=0\}$ with $Q_0,Q_1$ quadratic forms
in $x,y,z,w$, we have $\Gamma=\{Q_0=Q_1=0\}$. For general $[Y]\in\bP^{19}$ the curve
$\Gamma$ is an irreducible elliptic curve of degree $4$, and
\[
(C_2\cdot\xi)=0,\qquad (C_2\cdot\eta)=4 .
\]
Thus $C_2$ parametrizes the lines $\ell_p\subseteq Y$ mapping to points $p\in\Gamma$; in
particular the projection $C_2\to\Gamma$ is an isomorphism, and we use it to identify the
two curves below.

\begin{lemma}\label{lem:testing family 2}
    Let $\lambda_{\CM,f_2}$ be the CM line bundle of the family $f_2$. Then
    \[\deg\lambda_{\CM,f_2}
\ =\ (C_2\cdot\lambda_{\CM,f})
\ =\ -\bigl(K_{\sX_2/C_2}^4\bigr)
\ =\ 80.
    \]
\end{lemma}

\begin{proof}
    By construction $h_2\colon\mts{X}_2\to\mts{Y}_2:=Y\times C_2$ is the blow-up along
    ${\cL}_2$, which is the image of
    \[
    \bP^1\times C_2\xrightarrow{\ \Id_{\bP^1}\times\Delta\ }
    \bP^1\times C_2\times C_2\hookrightarrow\bP^1\times\bP^3\times C_2 ,
    \]
    where the second arrow is induced by $C_2\simeq\Gamma\subseteq\bP^3$ on the middle
    factor and $\Delta$ is the diagonal. Since $\Gamma\subseteq\bP^3$ has degree $4$,
    under this identification
    \[
    \cO_{\bP^1\times\bP^3\times C_2}(n_1,n_2,n_3)\big|_{{\cL}_2}
    \ \simeq\ \cO_{\bP^1\times C_2}(n_1,\,4n_2+n_3).
    \]
    The exact sequences
    \[
    0\longrightarrow N_{{\cL}_2/\bP^1\times C_2\times C_2}
    \longrightarrow N_{{\cL}_2/\bP^1\times\bP^3\times C_2}
    \longrightarrow N_{\bP^1\times C_2\times C_2/\bP^1\times\bP^3\times C_2}\big|_{{\cL}_2}
    \longrightarrow 0 ,
    \]
    \[
    0\longrightarrow N_2:=N_{{\cL}_2/\mts{Y}_2}
    \longrightarrow N_{{\cL}_2/\bP^1\times\bP^3\times C_2}
    \longrightarrow N_{\mts{Y}_2/\bP^1\times\bP^3\times C_2}\big|_{{\cL}_2}
    \longrightarrow 0 ,
    \]
    together with
    \[
    N_{{\cL}_2/\bP^1\times C_2\times C_2}\simeq\cO_{{\cL}_2},
    \quad
    N_{\bP^1\times C_2\times C_2/\bP^1\times\bP^3\times C_2}\simeq\cO^{\oplus2}_{\bP^1\times C_2\times C_2}(0,2,0),
    \quad
    N_{\mts{Y}_2/\bP^1\times\bP^3\times C_2}\simeq\cO_{\mts{Y}_2}(1,2,0),
    \]
    give
    \[
    c_1(N_2)=\cO_{\bP^1\times C_2}(-1,8),
    \qquad
    c_2(N_2)=0 .
    \]
    Writing $\zeta_2$ for $\cO_{\bP N_2^{*}}(1)$ on $E_2\simeq\bP N_2^{*}$, we obtain
    \[
    \zeta_2^{2}=-c_1(N_2)\zeta_2-c_2(N_2),
    \qquad
    \zeta_2^{3}=\big(c_1(N_2)^{2}-c_2(N_2)\big)\zeta_2=-16 .
    \]
    Since $K_{\mts{X}_2/C_2}=h_2^{*}q_1^{*}K_Y+E_2$, we conclude
    \[
    \begin{split}
    -\big(K_{\mts{X}_2/C_2}\big)^{4}
    &=-\big(h_2^{*}q_1^{*}\cO_Y(-1,-2)+E_2\big)^{4}\\
    &=-6\big(h_2^{*}q_1^{*}\cO_Y(1,2)^{2}\cdot E_2^{2}\big)
      +4\big(h_2^{*}q_1^{*}\cO_Y(1,2)\cdot E_2^{3}\big)-\big(E_2^{4}\big)\\
    &=6\big(q_1^{*}\cO_Y(1,2)^{2}\cdot{\cL}_2\big)
      -4\big(q_1^{*}\cO_Y(1,2)|_{{\cL}_2}\cdot c_1(N_2)\big)+\big(\zeta_2^{3}\big)\\
    &=6\cdot16-0-16\ =\ 80 . \qedhere
    \end{split} 
    \]
\end{proof}

\begin{corollary}\label{cor:CM slope}
    One has
    \[
    -f_{*}\big(-K_{\mts{X}/\bP\cE}\big)^{4}\ =\ 84\xi+20\eta .
    \]
    In particular, the CM line bundle $\lambda_{\CM,f}$ is proportional to
    $\xi+\tfrac{5}{21}\eta$, and is ample on $\bP\cE$.
\end{corollary}

\begin{proof}
   Writing $\lambda_{\CM,f}=a\xi+b\eta$, \Cref{lem:testing family 1} gives $a=84$, while \Cref{lem:testing family 2} gives $4b=80$. Hence
\[
    \lambda_{\CM,f}=84\xi+20\eta. \qedhere
\]
\end{proof}

Note that $t_0:=\tfrac{5}{21}<\tfrac27$, so the CM slope lies in the first chamber
$\big(0,\tfrac27\big)$ of \Cref{thm:wall crossing}.

\subsection{Deformation theory}

In this subsection we prove that the K-moduli stack
$\MM^{\textup{K}}_{\textup{\textnumero 3.11}}$ is smooth.

\begin{prop}\label{prop:no obstruction}
Let $Y\subseteq\bP^1\times\bP^3$ be a normal hypersurface of class
$\cO_{\bP^1\times\bP^3}(1,2)$, let
$\ell_p=\bP^1\times\{p\}\subset Y$ be a line such that $Y$ is
generically smooth along $\ell_p$, and let $X\coloneqq\Bl_{\ell_p}Y$.
Then \[\Ext_X^2(\Omega_X^1,\cO_X)=0.\]In particular, the deformations
of $X$ are unobstructed.
\end{prop}

\begin{proof}
Let
\[
    \psi\colon
    Z\coloneqq\Bl_{\ell_p}(\bP^1\times\bP^3)
    \ \longrightarrow\ \bP^1\times\bP^3
\]
be the blowup, with exceptional divisor $F$. Since $Y$ is generically
smooth along $\ell_p$, it has multiplicity one along $\ell_p$.
Consequently, $X=\Bl_{\ell_p}Y$ is naturally identified with the
strict transform of $Y$ in $Z$. In particular, $X$ is a Cartier
divisor in $Z$ of class
$\psi^*\cO_{\bP^1\times\bP^3}(1,2)-F$. We thus have a commutative
diagram
\[
\xymatrix{
    X\ar[rr]^{\phi}\ar@{^{(}->}[d]
    &&
    Y\ar@{^{(}->}[d]
    \\
    Z\ar[rr]^{\psi}
    &&
    \bP^1\times\bP^3.
}
\] Since $X$ is a Cartier divisor in the smooth fourfold $Z$, its
conormal sequence is
\[
    0\ \longrightarrow\  N_{X/Z}^*
   \  \longrightarrow\ \Omega_Z^1|_X
   \  \longrightarrow\ \Omega_X^1
   \  \longrightarrow\ 0.
\]
Applying $\RHom_X(-,\cO_X)$ reduces the claim to proving
\[H^1(X,N_{X/Z})=0,\qquad \Ext_X^2(\Omega_Z^1|_X,\cO_X)=0.\]

Set $\wt{\bP}^3\coloneqq\Bl_p\bP^3$. Since
$\ell_p=\bP^1\times\{p\}$, there is a natural isomorphism $Z\simeq\bP^1\times\wt{\bP}^3$. We continue to denote by $H$ the pullback of the hyperplane class on
$\bP^3$ and, by abuse of notation, by $F$ the exceptional divisor of
$\wt{\bP}^3\to\bP^3$ and its pullback to $Z$. Under the above
identification,
\[
    \cO_Z(X)
   \  \simeq\ 
    \cO_{\bP^1}(1)\boxtimes
    \cO_{\wt{\bP}^3}(2H-F).
\]
Since $2H-F$ is ample on $\wt{\bP}^3$, the divisor $X$ is ample on
$Z$. Moreover, $K_{\wt{\bP}^3}=-4H+2F$, and hence
$K_Z\sim-2X$. Then by adjunction, one has
\[
    N_{X/Z}\ \simeq\ \cO_X(X)\ \simeq\ \omega_X^{-1}.
\] The vanishing $H^1(X,N_{X/Z})=0$ follows from the restriction sequence
\[
    0\ \longrightarrow \ \cO_Z
   \  \longrightarrow\ \cO_Z(X)
     \ \longrightarrow \ N_{X/Z}
     \ \longrightarrow\ 0
\]
and Kodaira vanishing 
\[
    H^i(Z,\cO_Z(X))\ =\ H^i(Z,\cO_Z)\ =\ 0
    \qquad\text{for every }i>0.
\] For the second vanishing, since $\Omega_Z^1|_X$ is locally free, we
have
\[
    \Ext_X^2(\Omega_Z^1|_X,\cO_X)
    \ \simeq\  H^2(X,T_Z|_X).
\]
Consider the restriction sequence
\[
    0\ \longrightarrow \ T_Z(-X)
   \  \longrightarrow \ T_Z
   \  \longrightarrow\  T_Z|_X
   \  \longrightarrow\ 0.
\]
It suffices to show that $H^2(Z,T_Z)=0$ and
$H^3(Z,T_Z(-X))=0$. The product decomposition gives
\[
    T_Z\ \simeq\ 
    \operatorname{pr}_1^*T_{\bP^1}
    \oplus
    \operatorname{pr}_2^*T_{\wt{\bP}^3}.
\]
Consider the toric Euler sequence on $\wt{\bP}^3$ 
\[
    0\ \longrightarrow\ \cO_{\wt{\bP}^3}^{\oplus2}
    \ \longrightarrow \ 
    \cO_{\wt{\bP}^3}(H)
    \oplus\cO_{\wt{\bP}^3}(H-F)^{\oplus3}
    \oplus\cO_{\wt{\bP}^3}(F)
  \   \longrightarrow\  T_{\wt{\bP}^3}
   \ \longrightarrow\ 0.
\]
Then by Kawamata--Viehweg vanishing one has
\[
    H^2\bigl(\wt{\bP}^3,\cO_{\wt{\bP}^3}(H)\bigr)
   \  =\
    H^2\bigl(\wt{\bP}^3,\cO_{\wt{\bP}^3}(H-F)\bigr)
    \ =\ 0.
\]
Furthermore, the exact sequence
\[
    0\ \longrightarrow\ \cO_{\wt{\bP}^3}\ 
    \longrightarrow\ \cO_{\wt{\bP}^3}(F)
   \ \longrightarrow\ \cO_F(F)\simeq\cO_{\bP^2}(-1)
   \ \longrightarrow\ 0
\]
yields $H^2(\wt{\bP}^3,\cO_{\wt{\bP}^3}(F))=0$. Therefore, one obtains
$H^2(\wt{\bP}^3,T_{\wt{\bP}^3})=0$ and $H^2(Z,T_Z)=0$. Finally, one has
\[
\begin{aligned}
    T_Z(-X)
    \ \simeq\ \big(
    \cO_{\bP^1}(1)
    \boxtimes\cO_{\wt{\bP}^3}(-2H+F)\big)
    \ \oplus\ \big(
    \cO_{\bP^1}(-1)
    \boxtimes T_{\wt{\bP}^3}(-2H+F)\big).
\end{aligned}
\]
The second summand is acyclic because
$H^i(\bP^1,\cO_{\bP^1}(-1))=0$ for all $i$. For the first summand,
by Serre duality one has
\[
    H^3\bigl(
        \wt{\bP}^3,
        \cO_{\wt{\bP}^3}(-2H+F)
    \bigr)^*
   \ \simeq\ H^0\bigl(
        \wt{\bP}^3,
        \cO_{\wt{\bP}^3}(-2H+F)
    \bigr)
    \ =\ 0,
\]
and hence
$H^3(Z,T_Z(-X))=0$. The restriction sequence now gives
$H^2(X,T_Z|_X)=0$, and thus $\Ext_X^2(\Omega_Z^1|_X,\cO_X)=0$.
\end{proof}

\begin{corollary}\label{cor:smoothness}
    The K-moduli stack $\MM^{\textup{K}}_{\textup{\textnumero 3.11}}$ is smooth, and is a
    connected component of $\MM^{\textup{K}}_{3,28}$. Its good moduli space
    $\ove{\fM}^{\textup{K}}_{\textup{\textnumero 3.11}}$ is normal.
\end{corollary}

\subsection{K-moduli and K-stability of Fano threefolds in family \textnumero3.11}

\begin{theorem}\label{thm:K-moduli-of-3.11}
    Let $Y$ be a normal hypersurface in $\bP^1\times\bP^3$ of class
    $\cO_{\bP^1\times\bP^3}(1,2)$, let $\ell_p\subseteq Y$ be a line not contained in the
    singular locus of $Y$, and let $X:=\Bl_{\ell_p}Y$. Then the following are equivalent:
    \begin{itemize}
        \item $X$ is K-(semi/poly)stable;
        \item $[(Y,\ell_p)]\in\bP\cE$ is GIT (semi/poly)stable for the
        $\PGL(2)\times\PGL(4)$-action with respect to the linearized ample $\bQ$-line
        bundle $\xi+t_0\eta$, where $t_0=\tfrac{5}{21}$.
    \end{itemize}
    Moreover, there is a natural isomorphism of good moduli spaces
    \[
    \ove{\fM}^{\textup{K}}_{\textup{\textnumero 3.11}}
    \ \simeq\ \ove{\fM}^{\GIT}(t_0)
    =\bP\cE^{\sst}(t_0)\sslash_{t_0}\PGL(2)\times\PGL(4).
    \]
    In particular, if $X$ is K-semistable, then $Y$ is K-semistable.
\end{theorem}

\begin{proof}
    The identification of the K-moduli stack with the quotient stack is by now standard;
    see for instance \cite[Proof of Theorem 1.1]{LX19} or \cite[Theorem 5.10]{LZ25}.
\end{proof}

\begin{rem}
    The stacks  $\MM^\K_{\textup{\textnumero 3.11}}$ and $\MM^{\GIT}(t_0)$ are not isomorphic, since the stabilizer groups in $\MM^{\GIT}(t_0)$ are larger than those of the corresponding K-moduli points. The quotient-stack presentation of $\MM^\K_{\textup{\textnumero 3.11}}$ is instead obtained from the parameter space of pointed curves $(C,p)$ corresponding to the pairs
    $(Y,\ell_p)$, with the action of $\PGL(4)$.
\end{rem}

\begin{corollary}\label{cor:equivalence}
    Let $X$ be the blow-up of a normal hypersurface $Y\subseteq\bP^1\times\bP^3$ of class
    $\cO_{\bP^1\times\bP^3}(1,2)$ along a line $\ell_p$. Then the following are
    equivalent:
    \begin{enumerate}
        \item $X$ is K-semistable;
        \item $X$ is K-stable;
        \item $C_Y$ is a $(2,2)$-complete intersection curve which is either a smooth
        elliptic curve, or an irreducible nodal curve, or the union of a line and a
        twisted cubic glued at two distinct points with at worst nodal singularities;
        and moreover $p\in C_Y$ is a smooth point whose component in $C_Y$ has degree at
        least $3$.
    \end{enumerate}
    In particular, every K-semistable Fano threefold in the deformation family
    \textnumero 3.11 has at worst $A_1$-singularities.
\end{corollary}

\begin{proof}
    Since $t_0=\tfrac{5}{21}<\tfrac27$, this follows from
    \Cref{thm:K-moduli-of-3.11} together with \Cref{thm:wall crossing}(1).
\end{proof}

\begin{corollary}
    The K-moduli stack $\MM^{\textup{K}}_{\textup{\textnumero 3.11}}$ is a smooth proper
    Deligne--Mumford stack, and the K-moduli space
    $\ove{\fM}^{\textup{K}}_{\textup{\textnumero 3.11}}$ has quotient singularities.
\end{corollary}

\begin{proof}
    By \Cref{cor:equivalence} every K-semistable member is K-stable, so all
    automorphism groups are finite and the stack is Deligne--Mumford; smoothness is
    \Cref{cor:smoothness}. The statement on singularities follows since a smooth
    Deligne--Mumford stack has a good moduli space with quotient singularities.
\end{proof}
\smallskip

We conclude by placing our results in the broader context of smoothable Fano threefolds of anticanonical volume $28$. Their K-moduli space is now completely understood: it has five connected components, corresponding to
Mori--Mukai families \[\textup{\textnumero 2.21,\qquad \textnumero 3.11,\qquad
\textnumero 3.12, \qquad\textnumero 4.2,\qquad  \textnumero 5.1.}\]
The component corresponding to family~\textnumero 5.1 consists of a single
point. The one-dimensional component for family~\textnumero 3.12 is
described in~\cite{ACD+23}; the component for family~\textnumero 4.2,
whose members are Casagrande--Druel varieties, is studied in~\cite{CDF};
and the component for family~\textnumero 2.21 is described in~\cite{Mal26}.
These three components are obtained by a computational approach: one
constructs an explicit family of Fano threefolds over a projective parameter
space and verifies K-stability using the admissible flags introduced
in~\cite{AZ22}, together with estimates of $\delta$-invariants. By contrast,
the component for family~\textnumero 3.11 is described in the present paper
through the geometry of K3 surfaces, deformation theory and Hilbert scheme of curves.

Motivated by the deformation-theoretic perspective of~\cite{KLPZ26}, we
propose that these five families should not be viewed in isolation. Rather,
they should be studied simultaneously through their common Gorenstein
canonical degenerations, including those that are not necessarily
K-semistable. Although such degenerations do not themselves define points
of the K-moduli space, they may relate different deformation families and
thereby place the five K-moduli components in a unified geometric framework.

\newpage

\appendix
\section[tocentry={K-stability of an $A_1$ Fano degeneration}]
{K-stability of an $A_1$ Fano degeneration}
\label{appendix A}

Let $C \subset \bP^3$ be a smooth plane cubic, and let
$L \subset \bP^3$ be a line not contained in the plane spanned by $C$.
Assume that $L$ meets $C$ at a unique point $q$, where $C \cup L$ has a
simple node. Choose a point $p \in L \setminus C$. We first blow up
$\bP^3$ at $p$ and then blow up the strict transform of $C \cup L$.
The resulting threefold is a terminal Gorenstein Fano threefold with a
unique ordinary double point and anticanonical volume $28$. It is a
boundary member of Mori--Mukai Family~$3.11$.

\begin{prop}\label{prop:K-unstable}
The one-nodal Fano threefold $X$ constructed above is K-unstable.
More precisely, there exists a divisor $D$ over $X$ such that
\[
    \beta_X(D) = -\frac{5}{56} < 0.
\]
\end{prop}

We begin by describing a small resolution of $X$ that will be used in
the proof. Starting with $\bP^3$, $p$, $L$, and $C$ as above, consider
the following sequence of three blowups along smooth centers:
\begin{enumerate}
    \item First, blow up the point $p \in \bP^3$, and denote the
    resulting morphism by
    \[
        \pi_1 \colon \wt{\bP}^3 \longrightarrow \bP^3.
    \]
    \item Next, blow up the strict transform of $L$, and denote the
    resulting morphism by
    \[
        \pi_2 \colon \widehat{X} \longrightarrow \wt{\bP}^3.
    \]
    \item Finally, blow up the strict transform of $C$, and denote the
    resulting morphism by
    \[
        \pi_3 \colon \widetilde{X} \longrightarrow \widehat{X}.
    \]
\end{enumerate}
These morphisms fit into the following commutative diagram:
\[
\begin{tikzcd}
    \widetilde{X}
        \arrow[r, "\pi_3"]
        \arrow[d, "\vartheta"]
    & \widehat{X}
        \arrow[rd, "\pi_2"]
    & {} \\
    X
        \arrow[rr, "\varphi"']
    & {}
    & \wt{\bP}^3
        \arrow[r, "\pi_1"]
    & \bP^3.
\end{tikzcd}
\]
By construction, $\widetilde{X}$ is smooth, and the morphism
$\vartheta$ is an isomorphism away from a locus of codimension two.
Since the small resolution $\vartheta$ is crepant, $\widetilde{X}$ is
a weak Fano threefold.

We introduce some notation. Let
\[
    \pi \coloneqq \pi_1 \circ \pi_2 \circ \pi_3
    \colon \widetilde{X} \longrightarrow \bP^3,
    \qquad
    H \coloneqq \pi^*\mathcal{O}_{\bP^3}(1).
\]
Let $F$ be the exceptional divisor of $\pi_1$, and denote by $F'$ and
$F''$ its strict transforms on $\widehat{X}$ and $\widetilde{X}$,
respectively. Similarly, let $E_L$ be the exceptional divisor of
$\pi_2$, and let $E_L'$ be its strict transform on $\widetilde{X}$.
Finally, let $E_C$ be the exceptional divisor of $\pi_3$.

\begin{lem}
On $\widetilde{X}$, the following intersection numbers hold:
\[
\begin{array}{cccc}
    H^3 = 1
    & H^2 \cdot E_L' = 0
    & H^2 \cdot E_C = 0
    & H^2 \cdot F'' = 0
    \\[2pt]
    E_C^3 = -11
    & E_C^2 \cdot H = -3
    & E_C^2 \cdot E_L' = -1
    & E_C^2 \cdot F'' = 0
    \\[2pt]
    (E_L')^3 = 0
    & (E_L')^2 \cdot H = -1
    & (E_L')^2 \cdot E_C = 0
    & (E_L')^2 \cdot F'' = -1
    \\[2pt]
    (F'')^3 = 1
    & (F'')^2 \cdot H = 0
    & (F'')^2 \cdot E_C = 0
    & (F'')^2 \cdot E_L' = 0
    \\[2pt]
    H \cdot E_L' \cdot E_C = 0
    & H \cdot E_L' \cdot F'' = 0
    & E_L' \cdot E_C \cdot F'' = 0
    & H \cdot E_C \cdot F'' = 0.
\end{array}
\]
\end{lem}

\begin{proof}
First, observe that $E_C$ and $F''$ are disjoint. Hence every
intersection product involving both divisors vanishes. Moreover,
\[
    H^3 = \mathcal{O}_{\bP^3}(1)^3 = 1,
\]
while all the remaining products involving $H^2$ vanish for
dimensional reasons. The same argument, together with the projection
formula, gives
\[
    H \cdot (F'')^2
    = \pi_1^*\mathcal{O}_{\bP^3}(1) \cdot F^2
    = 0.
\]

Let $C'' \subset \widehat{X}$ denote the strict transform of $C$.
Then
\[
    \pi_2^*\pi_1^*\mathcal{O}_{\bP^3}(1) \cdot C'' = 3,
    \qquad
    E_L \cdot C'' = 1.
\]
It follows that
\[
    E_C^2 \cdot H = -3,
    \qquad
    E_C^2 \cdot E_L' = -1.
\]
Furthermore,
\[
    -K_{\widehat{X}}
    = \pi_2^*\pi_1^*\mathcal{O}_{\bP^3}(4)-2F'-E_L,
\]
and hence
\[
    -K_{\widehat{X}} \cdot C'' = 11.
\]
Since $C''$ is an elliptic curve, the adjunction formula gives
\[
    \deg N_{C''/\widehat{X}}
    = -K_{\widehat{X}} \cdot C''
    = 11.
\]
Therefore,
\[
    E_C^3 = -\deg N_{C''/\widehat{X}} = -11.
\]

Let $L' \subset \wt{\bP}^3$ denote the strict transform of
$L$. We have
\[
    \pi_1^*\mathcal{O}_{\bP^3}(1) \cdot L' = 1,
    \qquad
    F \cdot L' = 1.
\]
Consequently,
\[
    (E_L')^2 \cdot H = -1,
    \qquad
    (E_L')^2 \cdot F'' = -1.
\]
Moreover,
\[
    -K_{\wt{\bP}^3} \cdot L'
    =
    \bigl(\pi_1^*\mathcal{O}_{\bP^3}(4)-2F\bigr)\cdot L'
    = 2.
\]
Since $L'$ is a smooth rational curve, adjunction yields
\[
    \deg N_{L'/\wt{\bP}^3}
    = -K_{\wt{\bP}^3}\cdot L'-2
    = 0.
\]
Thus,
\[
    (E_L')^3 = E_L^3
    = -\deg N_{L'/\wt{\bP}^3}
    = 0.
\] Finally, $(F'')^3=F^3=1$, and all the remaining intersection products
vanish by the projection formula and dimensional considerations.
\end{proof}

\begin{proof}[Proof of Proposition~\ref{prop:K-unstable}]
Let $D$ be the strict transform on $\widetilde{X}$ of the plane
containing $C$. It defines a prime divisor over $X$, and we will show
that it destabilizes $X$. We have
\[
    D \equiv H-E_C
    \qquad\text{and}\qquad
    A_X(D)=1.
\]
Recall from Definition~\ref{defn:beta} that
\[
    S_X(D)
    =
    \frac{1}{\vol(-K_X)}
    \int_0^\tau
        \vol\bigl(-K_{\widetilde{X}}-tD\bigr)\,dt,
\]
where $\tau$ is the pseudoeffective threshold of $D$ with respect to
$-K_X$. Set
\[
    D_t \coloneqq -K_{\widetilde{X}}-tD.
\]
Numerically,
\[
    D_t
    \equiv
    (4-t)H-(1-t)E_C-E_L'-2F''.
\]

We make the following observations:
\begin{itemize}
    \item The divisor $D_0=-K_{\widetilde{X}}$
    is nef.
    \item The linear system of quadric cones in $\bP^3$ with vertex
    $p$ and containing $L$ has its base locus resolved by
    $\pi_1\circ\pi_2$. Hence
    \[
        2H-E_L'-2F''
    \]
    is nef and movable. Its top self-intersection, and therefore its
    volume, is zero.

    \item The divisor
    \[
        D_1=3H-E_L'-2F''
        =H+\bigl(2H-E_L'-2F''\bigr)
    \]
    is nef and movable. It follows that $D_t$ is nef for every
    $t\in[0,1]$.

    \item Every divisor on the segment joining $D_1$ and
    $2H-E_L'-2F''$ is nef and movable.
\end{itemize} We are therefore able to deduce the Zariski decomposition of $D_t$. For $t\in[0,1]$, we have
\[
    N(t)=0,
    \qquad
    P(t)=D_t,
\]
because $D_t$ is nef on this interval. For $t\ge 1$, the negative part clearly must contain $(t-1)E_C$. Moreover, we have just observed that $D_t - (t-1)E_C$ is nef for $t\in[1,2]$. To recapitulate, for $t\in [1,2]$ we have:
\[
    N(t)=(t-1)E_C,
    \qquad
    P(t)\equiv(4-t)H-E_L'-2F''.
\]
Finally, we have $\vol(D_2) = \vol(P(2)) = 0$, so the pseudoeffective threshold is $\tau=2$. Using the
intersection numbers computed above, we obtain
\begin{align*}
    S_X(D)
    &=
    \frac{1}{28}\int_0^2 P(t)^3\,dt
    \\
    &=
    \frac{1}{28}
    \left(
        \int_0^1
        \bigl((4-t)H-(1-t)E_C-E_L'-2F''\bigr)^3\,dt
        +
        \int_1^2
        \bigl((4-t)H-E_L'-2F''\bigr)^3\,dt
    \right)
    \\
    &=
    \frac{1}{28}
    \left(
        \int_0^1
        \bigl(28-3t-6t^2-3t^3\bigr)\,dt
        +
        \int_1^2
        (2-t)(5-t)^2\,dt
    \right)
    \\
    &=
    \frac{1}{28}
    \left(
        \frac{95}{4}+\frac{27}{4}
    \right)
    =
    \frac{61}{56}.
\end{align*}
Consequently,
\[
    \beta_X(D)
    =
    A_X(D)-S_X(D)
    =
    1-\frac{61}{56} = -\frac{5}{56}
    <0.
\]
We have therefore shown that $D$ destabilizes $X$ -- by Theorem-Definition~\ref{theodef:K-stability} $X$ is K-unstable. 
\end{proof}

\bibliographystyle{alpha}
\bibliography{citation}

\end{document}